\documentclass[12pt,a4paper,reqno]{amsart}
\usepackage[english]{babel}
\usepackage{amssymb,latexsym,amsfonts,amsthm,upref,amsmath}
\usepackage[foot]{amsaddr}
\usepackage[margin=1in]{geometry} 
\usepackage[dvipsnames,x11names]{xcolor}
 \definecolor{myblue}{HTML}{003399}
\usepackage{hyperref}
\hypersetup{colorlinks,citecolor=myblue,filecolor=black,linkcolor=myblue,urlcolor=myblue}
\usepackage{enumerate}  
\usepackage{tikz}
\usepackage{float}
\usepackage{cleveref}
\usepackage{mathtools}
\usepackage{cite}
\usepackage{filecontents}
\usepackage{enumitem}
\makeatletter
\newcommand{\leqnomode}{\tagsleft@true}
\newcommand{\reqnomode}{\tagsleft@false}

\makeatother
\newtheorem*{thm*}{Theorem}
\newtheorem*{lem*}{Lemma}
\newtheoremstyle{prim}{}{}{\normalfont}{}{\bfseries}{.}{ }{}
\newtheoremstyle{stil}{}{}{\slshape}{}{\bfseries}{.}{ }{}
\theoremstyle{stil}
\newtheorem{thm}{Theorem}[section]
\newtheoremstyle{defi}{}{}{}{}{\bfseries}{.}{ }{}
\theoremstyle{defi}
\newtheorem{defn}[thm]{Definition}
\theoremstyle{defi}

\theoremstyle{stil}
\newtheorem*{mthm*}{Main Theorem}
\newtheorem*{kor*}{Corollary}

\theoremstyle{stil}
\newtheorem{lem}[thm]{Lemma}
\theoremstyle{stil}
\newtheorem{kor}[thm]{Corollary}
\theoremstyle{prim}
\newtheorem{ex}[thm]{Example}
\newenvironment{prf}{\noindent \textit{Proof.}}{\null\hfill$\qed$\hskip
2mm\vskip 2mm}

\newcommand{\DYeven}{{\rm DY} (\mathfrak{gl}_{M })}

\newcommand{\dygext}{ {\rm DY}(\mathfrak{gl}_{M|N})^{ext}}
\newcommand{\dygextfin}{ {\rm DY}(\mathfrak{gl}_{M|N})_{fin}^{ext}}
\newcommand{\dyg}{ {\rm DY}(\mathfrak{gl}_{M|N})}
\newcommand{\dygc}{ {\rm DY}_c(\mathfrak{gl}_{M|N})}

\newcommand{\dygtld}{ \wtld{\rm DY}_{crit}(\mathfrak{gl}_{M|N})}
\newcommand{\dygctld}{ \wtld{\rm DY}_{c}(\mathfrak{gl}_{M|N})}
\newcommand{\dygNMtld}{ \wtld{\rm DY}_{N-M}(\mathfrak{gl}_{M|N})}

\newcommand{\Y}{ {\rm Y}(\mathfrak{gl}_{M|N})}
\newcommand{\Yd}{ {\rm Y}^+(\mathfrak{gl}_{M|N})} 
\newcommand{\Ydext}{ {\rm Y}^+(\mathfrak{gl}_{M|N})^{ext}} 

\newcommand{\U}{ {\rm U}}

\newcommand{\B}{ {\rm B}(\mathfrak{gl}_{M|N})}
\newcommand{\Bd}{ {\rm B}^+(\mathfrak{gl}_{M|N})}
\newcommand{\DB}{ {\rm DB} (\mathfrak{gl}_{M|N})} 
\newcommand{\DBB}{  \mathcal{DB}  (\mathfrak{gl}_{M|N})} 
\newcommand{\DBBc}{  \mathcal{DB}_c  (\mathfrak{gl}_{M|N})} 
\newcommand{\DBBctld}{  \widetilde{\mathcal{DB}}_c  (\mathfrak{gl}_{M|N})} 
\newcommand{\DBBcrittld}{  \widetilde{\mathcal{DB}}_{crit}  (\mathfrak{gl}_{M|N})} 
\newcommand{\DBBNMtld}{  \widetilde{\mathcal{DB}}_{(N-M)/2}  (\mathfrak{gl}_{M|N})} 
\newcommand{\BB}{ \mathcal{B} (\mathfrak{gl}_{M|N})}
\newcommand{\BBd}{ \mathcal{B}^+(\mathfrak{gl}_{M|N})}
 \newcommand{\BBdc}{ \mathcal{B}_c^+(\mathfrak{gl}_{M|N})_h}
   \newcommand{\BBdNM}{ \mathcal{B}_{(N-M)/2}^+(\mathfrak{gl}_{M|N})_h}
   \newcommand{\BBdcrit}{ \mathcal{B}_{crit}^+(\mathfrak{gl}_{M|N})_h}
\newcommand{\R}{ {\overline{R}}}

\newcommand{\vac}{\mathop{\mathrm{\boldsymbol{1}}}}

\newcommand{\gr}{\mathop{\mathrm{gr}}}

\newcommand{\glmn}{\mathfrak{gl}_{M|N}}
\newcommand{\glmnht}{\widehat{\mathfrak{gl}}_{M|N}}
\newcommand{\gl}{\mathfrak{gl}}

\newcommand{\CC}{\mathbb{C}}

\newcommand{\ZZ}{\mathbb{Z}}

\newcommand{\Sc}{\mathcal{S}}

\newcommand{\Dc}{\mathcal{D}}

\newcommand{\wtld}{\widetilde}

\newcommand{\wndr}{\underline}

\newcommand{\ot}{\otimes}
\newcommand{\ts}{\hspace{1pt}}

\newcommand{\str}{ \mathop{\rm str}}

\newcommand{\ndo}{\mathop{\mathrm{End}}}
\newcommand{\om}{\mathop{\mathrm{Hom}}}

\newcommand{\diag}{\mathop{\mathrm{diag}}}

\newcommand{\fand}{\quad\text{and}\quad}
\newcommand{\Fand}{\qquad\text{and}\qquad}

\newcommand{\non}{\nonumber}
\newcommand{\beq}{\begin{equation}}
\newcommand{\eeq}{\end{equation}}
\newcommand{\ben}{\begin{equation*}}
\newcommand{\een}{\end{equation*}}

\makeatletter
\def\smalloverbrace#1{\mathop{\vbox{\m@th\ialign{##\crcr\noalign{\kern3\p@}%
  \tiny\downbracefill\crcr\noalign{\kern3\p@\nointerlineskip}%
  $\hfil\displaystyle{#1}\hfil$\crcr}}}\limits}
\makeatother

\makeatletter
\def\smallunderbrace#1{\mathop{\vtop{\m@th\ialign{##\crcr
   $\hfil\displaystyle{#1}\hfil$\crcr
   \noalign{\kern3\p@\nointerlineskip}%
   \tiny\upbracefill\crcr\noalign{\kern3\p@}}}}\limits}
\makeatother

\begin{document}

\title{Double Yangian and  reflection algebras of the Lie superalgebra $\mathfrak{gl}_{M|N}$, II: Quantum currents}

\author{Lucia Bagnoli}
\address[L. Bagnoli]{Dipartimento di Matematica, Sapienza Universit\`{a} di Roma, P.le Aldo Moro 5, 00185 Rome, Italy \& INFN sezione di Roma}
\email{lucia.bagnoli@uniroma1.it}

\author{Slaven Ko\v{z}i\'{c}}
\address[S. Ko\v{z}i\'{c}]{University of Zagreb Faculty of Science, Department of Mathematics, Bijeni\v{c}ka cesta 30, 10000 Zagreb, Croatia}
\email{slaven.kozic@math.hr}

\begin{abstract}
In this paper, we continue our research on the double Yangian and  reflection algebras of the Lie superalgebra $\mathfrak{gl}_{M|N}$. 
 Extending the Etingof--Kazhdan construction, we introduce 
the structure of $h$-adic quantum vertex superalgebra on
 the vacuum module $V^c(\mathfrak{gl}_{M|N})$ over the double Yangian  of level   $c\in\CC$. 
Next, we construct families of central elements in the completed double Yangian for $\mathfrak{gl}_{M|N}$ at the critical level. Finally, we show that the level $c$ restricted modules over the reflection algebra of   $\mathfrak{gl}_{M|N}$ are naturally equipped with the structure of quasi $V^{2c}(\mathfrak{gl}_{M|N})$-module. By using this structure, we find explicit formulas for families of central elements in the completed reflection algebra of   $\mathfrak{gl}_{M|N}$ at the critical level.
\end{abstract}

\maketitle

\allowdisplaybreaks
\section{Introduction}\label{section_1}
\numberwithin{equation}{section}

The   Yangian $\Y$ of the Lie superalgebra $\mathfrak{gl}_{M|N}$ was introduced by Nazarov \cite{Naz} in terms of the $R$-matrix presentation.
It has been extensively  studied for its mathematical properties (see, e.g., the papers \cite{LGow,LM,T, M, P} and references therein), as well as for its connections with various areas of mathematical physics \cite{AK,JWW,CR,HY}.
Its centrally extended double, the double Yangian  $\dyg$ of the Lie superalgebra   $\mathfrak{gl}_{M|N}$,
 was introduced by  Zhang \cite{Z}, also
 in terms of the $R$-matrix presentation.
Recently, the double Yangians of the Lie superalgebras $\mathfrak{gl}^{\mathfrak{s}}_{M|N}$ and $\mathfrak{sl}^{\mathfrak{s}}_{M|N}$, associated with
an  arbitrary $0^M 1^N$-sequence $\mathfrak{s}$, were introduced and studied
by  Xu, Lin and Zhang \cite{XLZ}.

In this paper, we continue our investigation \cite{BK,BK2} of the double Yangian $\dyg$ and its reflection equation subalgebras. The main motivation for this research comes from the   quantum vertex algebra theory. The notion of quantum vertex algebra, as well as its first examples, which were associated with the rational, trigonometric and elliptic $R$-matrices of type $A$, goes back to Etingof and Kazhdan \cite{EK}. The rational case employs   the double Yangian $\DYeven$ of the Lie algebra $\mathfrak{gl}_{M }$ to construct the quantum vertex algebra structure over its vacuum module $V^c(\mathfrak{gl}_{M })$ of level $c\in\CC$.
The $\mathcal{S}$-locality  in   $V^c(\mathfrak{gl}_{M })$, a quantum analogue of the locality axiom from vertex algebra theory, takes the form of the quantum current commutation relation of Reshetikhin and Semenov-Tian-Shansky \cite{RS}. 
Later on, different aspects of the quantum vertex algebra theory,  closely related to the double Yangians, were studied by several authors; see, e.g,  \cite{DGK,Li,JKMY}. Recently, $h$-adic quantum vertex algebras were associated with double Yangians of all types, via their Drinfeld current type presentation, and the corresponding lattice $h$-adic quantum vertex algebra construction was established by Kong and Li \cite{KL1,KL2}.

In this paper, we consider the double Yangian $\dyg$ in its
  $R$-matrix presentation,
\begin{align*}
R_{12}(u-v)\ts T_1^\pm (u)\ts T_2^\pm (v)
&= T^\pm_2(v)\ts T^\pm_1(u)\ts R_{12}(u-v), \\
\R_{12}(u-v+C/2)\ts T^-_1 (u)\ts T_2^+ (v)
&= T_2^+ (v)\ts T^-_1 (u)\ts \R_{12}(u-v-C/2); 
\end{align*}
see Section \ref{section_2} for a precise formulation.
In this setting, the aforementioned Reshetikhin--Semenov-Tian-Shansky quantum current, which plays an important role throughout the entire paper,
takes the (usual) form 
\beq\label{quantumcurrent}
\mathcal{T}(u) = T^+(u) \ts T^-(u+C/2)^{-1}.
\eeq

Moreover, we consider the reflection equation subalgebra $\DB$ of the double Yangian.  The algebras associated with the reflection equation were    introduced by Sklyanin \cite{S} to describe integrable systems with the boundary conditions; see  \cite{GS,KS,KJC,MRS} for more information on such algebras and their applications. 
The algebra $\DB$,   investigated in this paper, is based on   the reflection algebras introduced by Molev and Ragoucy \cite{MR}.
In Section \ref{section_5}, we prove its $R$-matrix presentation in terms of generators which are subject to three reflection equations,
\begin{align*}
R (u-v)\ts B_1^\pm (u)\ts  R (u+v )\ts B_2^\pm (v)
&= B^\pm_2(v)\ts  R (u+v )\ts B^\pm_1(u)\ts R (u-v),   \\
 \R (u-v +C)\ts B^-_1 (u)\ts  \R (u+v -C)\ts B_2^+ (v)
&= B_2^+ (v)\ts  \R (u+v+C )\ts B^-_1 (u)\ts  \R (u-v -C),   
\end{align*}
and two unitarity equations; see Section \ref{section_5} for a precise formulation. This
  extends our results from \cite{BK}, which were obtained for $C=0$.

In Section \ref{section_4}, we consider  the  vacuum module $V^c(\mathfrak{gl}_{M|N})$ of level $c\in\CC$ over the double Yangian $\dyg$. 
We show that it can be equipped with the structure of $h$-adic quantum vertex superalgebra, such that every restricted $\dyg$-module of level $c\in\CC$ is also a $V^c(\mathfrak{gl}_{M|N})$-module.
 This can be regarded as a supersymmetric counterpart of the Etingof--Kazhdan construction of the quantum affine vertex algebra  $V^c(\mathfrak{gl}_{M })$ \cite{EK}.
As with the even case,   the   $\Sc$-locality in $V^c(\mathfrak{gl}_{M|N})$ possesses the form of the   quantum current commutation relation, which is now $\ZZ/2\ZZ$-graded. 

In Section \ref{section_3}, we construct families of central elements
in a certain completion of the double Yangian $\dyg$ at the critical level $N-M$.
This generalizes the construction of Jing, Molev, Yang and the second author \cite{JKMY}, which was obtained in the even case, i.e., for the double Yangian $\DYeven$.
The families are parametrized by Young tableaux, and their construction relies  on the fusion procedure for the symmetric group, which goes back to Jucys \cite{Juc}, while their form utilizes the normally ordered product of quantum currents \eqref{quantumcurrent}. 
Finally, we employ this result to find explicit formulas for the families of central elements  in the $h$-adic quantum vertex superalgebra $V^c(\mathfrak{gl}_{M|N})$ at the critical level $c=N-M$.
The trigonometric counterpart of this  construction, in the setting of   quantum affine superalgebra  $\mathrm{U}_q(\widehat{\mathfrak{gl}}_{M|N})$, was recently obtained by Jing, Liu and Zhang \cite{JLZ}.

In Section \ref{section_last}, we derive  supersymmetric counterparts of the  results from \cite{Koz}.
We use the analogue of the quantum current operator  \eqref{quantumcurrent} for the reflection algebra $\DB$,  
\beq\label{quantumcurrent2}
\mathcal{B}(u) = B^+(u) \ts B^-(u+C )^{-1},
\eeq 
 to prove that  every restricted $\DB$-module of level $c\in\CC$ is   equipped with the structure of quasi 
$V^{2c}(\mathfrak{gl}_{M|N})$-module.
Furthermore, by employing the corresponding quasi module map, we extend the construction of families of central elements from Section \ref{section_3} to the suitably completed reflection equation algebra at the critical level. As before, the construction relies on the fusion procedure, while the explicit formulas for the central elements employ  the form of   normally ordered product of quantum currents \eqref{quantumcurrent2}.

\section{Double Yangian for \texorpdfstring{$\glmn$}{glm|n}}\label{section_2}
 
In this section, we follow the exposition in \cite[Sect. 2]{BK} to recall the double Yangian $\dyg$ for the Lie superalgebra $\mathfrak{gl}_{M|N}$. 
Let $\ZZ_2=\ZZ / 2\ZZ =\left\{\bar{0},\bar{1}\right\}$.
Throughout the paper, we assume that $M\neq N$.
Consider the Lie superalgebra 
$$\glmnht=\glmn\ot\CC[t,t^{-1}]\oplus\CC K .$$
Let
$e_{ij}\in \glmn$ be the  matrix units and  $e_{ij}(r)=e_{ij}\ot t^r$ for all $r\in\ZZ$. The 
   parity of
the element
$e_{ij}(r)$ is defined to be $\bar{i}+\bar{j}$, where  $\bar{i}=0$ (resp. $\bar{i}=1$) for $i=1,\ldots ,M$ (resp. $i=M+1,\ldots ,M+N$). Moreover, the element $K$ is defined to be even and central. 
The supercommutation relations in $\glmnht$ are given by
\begin{align*}
[e_{ij}(r),e_{kl}(s)]
=&\, 
\delta_{kj}\ts  e_{il}(r+s)
-  \delta_{il}\ts e_{kj}(r+s) (-1)^{(\bar{i}+\bar{j})(\bar{k}+\bar{l})} \\
&+K\left(  \delta_{kj}\ts \delta_{il}(-1)^{\bar{i}} -\frac{\delta_{ij}\ts \delta_{kl}}{M-N}\ts (-1)^{\bar{i}+\bar{k}}\right)\ts r\ts \delta_{r+s\, 0}.
\end{align*}

The defining relations for the double Yangian can be expressed in terms of the  rational $R$-matrix
  $R(u)\in\ndo\CC^{M|N}\ot \ndo\CC^{M|N} [u^{-1}]$,   given by
\beq\label{R}
R(u)=I-Pu^{-1},\qquad\text{where}\qquad P=\sum_{i,j=1}^{M+N} e_{ij}\ot e_{ji}\ts (-1)^{\bar{j}}
\eeq
and $I$ is the identity. The $R$-matrix $R(u)$ satisfies the {\em quantum Yang--Baxter equation},
\beq\label{ybe}
R_{12}(u )\ts R_{13}(u+v)\ts R_{23}(v)=
R_{23}(v)\ts R_{13}(u+v)\ts R_{12}(u ).
\eeq

Let $g(u) $ be a unique formal power series
\beq\label{function_g}
g(u)=1+\sum_{l\geqslant 1} \frac{g_l}{u^l}\in 1+u^{-1}\CC[[u^{-1}]]
\eeq
such that
$$
g(u+M-N)=\left(1-u^{-2}\right)g(u);
$$ 
see \cite[Sect. 2.2]{JKMY} for more information. The normalized $R$-matrix 
\beq\label{Rbar}
\R(u)=g(u)R(u)
\eeq
satisfies the Yang--Baxter equation \eqref{ybe}, the {\em unitarity condition},
$$
\R(u) \ts\R(-u)=1,
$$ 
  and the {\em crossing symmetry identities},
\beq\label{csym}
\left(\R(u-M+N)^{-1}\right)^{\tau_i}\ts \R(u )^{\tau_i}=I\qquad\text{for }i=1,2,
\eeq
where
$\tau_i$ is the action of the   supertransposition
 \beq\label{supertransposition}
\tau\colon e_{rs}\mapsto (-1)^{\bar{r}\bar{s}+\bar{r}}e_{sr}, \quad  
r,s=1,\ldots ,M+N,
 \eeq
     on the 
$i$-th tensor factor.

The
{\em double Yangian $\dyg$ for $\mathfrak{gl}_{M|N}$} is defined as the $\ZZ_2$-graded unital associative algebra over the complex field generated by the elements $C$ and $t_{ij}^{(\pm r)}$  with $  i,j=1,\ldots , M+N$ and $r=1,2,\ldots,$ 
where $C$ is   even and central   and  the parity of
  $t_{ij}^{(\pm r)}$  is  $\bar{i}+\bar{j}$.
The generators are
subject to the  defining relations which are written in terms of the  matrices 
$$
T^\pm (u)=\sum_{i,j=1}^{M+N} (-1)^{\bar{i}\bar{j}+\bar{j}} e_{ij}\ot t^\pm_{ij}(u),
 $$
 where
$$
t_{ij}^-(u)=\delta_{ij}+\sum_{r\geqslant 1} t_{ij}^{(r)} u^{-r}
\Fand 
t_{ij}^+ (u)=\delta_{ij}-\sum_{r\geqslant 1} t_{ij}^{(-r)} u^{r-1}.
$$
The relations  are given by  
\begin{align}
R_{12}(u-v)\ts T_1^\pm (u)\ts T_2^\pm (v)
&= T^\pm_2(v)\ts T^\pm_1(u)\ts R_{12}(u-v), \label{rtt12}\\
\R_{12}(u-v+C/2)\ts T^-_1 (u)\ts T_2^+ (v)
&= T_2^+ (v)\ts T^-_1 (u)\ts \R_{12}(u-v-C/2). \label{rtt3}
\end{align}

Define the degree operator on the  double Yangian by
\beq\label{degreeoperator}
 \deg  t_{ij}^{(r)}=r-1 ,\qquad
\deg  t_{ij}^{(-r)}=-r
\Fand
\deg C=0.
\eeq
It induces the ascending filtration   
$$
\ldots \subseteq\dyg^{(r)}\subseteq \dyg^{(r+1)}\subseteq\ldots\subseteq\dyg,
$$
where $\dyg^{(r)}$ is the linear span of the elements of $\dyg$ whose degrees do not exceed $r$.    
Let  $\gr   \dyg$ be the corresponding graded algebra and  denote the images of   generators in their respective components by $\bar{t}_{ij}^{(\pm r)}$ and $\bar{C}$.
By the Poincar\'{e}--Birkhoff--Witt Theorem for the double Yangian \cite[Thm. 2.9]{BK}, 
 the assignments
$$
e_{ij}(r-1)\mapsto (-1)^{\bar{i}}\ts \bar{t}_{ij}^{(r)},\qquad e_{ij}(-r)\mapsto (-1)^{\bar{i}}\ts \bar{t}_{ij}^{(-r)}   \Fand K\mapsto \bar{C} 
$$ 
 define   an isomorphism of $\mathbb{Z}_2$-graded algebras
\beq\label{isomorphism}
\U(\glmnht) \to \textstyle\gr \dyg,
\eeq
where $\U(\glmnht)$ is the universal enveloping algebra of $\glmnht$.

Another consequence of the   Poincar\'{e}--Birkhoff--Witt Theorem is that the subalgebra of $\dyg$
generated by all elements $t_{ij}^{(r)}$ (resp. $t_{ij}^{(-r)}$) with $r=1,2,\ldots$ is isomorphic to the {\em Yangian} $\Y$ (resp. the {\em dual Yangian}  $\Yd$).
Hence, \eqref{isomorphism} yields the isomorphisms
$$
\U(\glmn[t] ) \to \textstyle\gr \Y\Fand
\U(t^{-1}\glmn[t^{-1}]) \to \textstyle\gr \Yd.
$$
 Finally, define the {\em vacuum module $V^c(\glmn)$ of the level $c\in\CC$} as the quotient of the double Yangian $\dyg$ by its left ideal generated by $C-c\cdot 1$ and all elements $t_{ij}^{(r)}$ with   $r=1,2,\ldots .$ Clearly, as a complex vector space,   $V^c(\glmn)$ is isomorphic to $\Yd$.

\section{Reflection algebras for \texorpdfstring{$\glmn$}{glm|n}}\label{section_5}
In this section, we study the  centrally extended reflection  algebras, a certain class of subalgebras of the suitably completed double Yangian.
Our main goal is to  extend the presentation of these algebras  at the level zero, which was established in   \cite[Sect. 3]{BK}, to their central extension. 

\subsection{Reflection subalgebras of the double Yangian} 

 Consider the descending filtration on the dual Yangian $\Yd$ defined by setting the degree of $t_{ij}^{(-r)}$ with $r\geqslant 1$ to be equal to $r$. Define the {\em extended dual Yangian} $\Ydext$ as the   completion of $\Yd$ with respect to this filtration and the {\em extended double Yangian} $\dygext$ as the space of all finite $\CC[C]$-linear combinations of all products $xy$ for $x\in \Ydext$ and $y\in \Y$ with the multiplication extended by continuity from the double Yangian.
Extend the degree operator \eqref{degreeoperator}  to the algebra $\dygext$ by allowing it to take the infinite value. 
Then the
elements of finite degree form a subalgebra which we denote by $\dygextfin$.

Fix an integer $1\leqslant \ell\leqslant M+N$. Let $G=(g_{ij}) $ be the diagonal matrix 
\beq\label{gmatricca}
G =\diag(\varepsilon_1,\ldots ,\varepsilon_{M+N}),
\quad\text{where}\quad\varepsilon_1=\ldots =\varepsilon_\ell  =-\varepsilon_{\ell+1}=\ldots =-\varepsilon_{M+N}=1.
\eeq 
Consider the matrices 
$$
B^+(u) = T^+(u)\ts G\ts T^+(-u)^{-1}
\Fand
B^-(u) = T^-(u+C/2)\ts G\ts T^-(-u+C/2)^{-1}
,
$$
which belong to
$ \ndo\CC^{M|N}\ot \dygextfin[[u^{\pm 1}]]$.
They satisfy the {\em reflection relations}
\begin{align}
R (u-v)\ts B_1^\pm (u)\ts  R (u+v )\ts B_2^\pm (v)
&= B^\pm_2(v)\ts  R (u+v )\ts B^\pm_1(u)\ts R (u-v), \label{reflect12} \\
 \R (u-v +C)\ts B^-_1 (u)\ts  \R (u+v -C)\ts B_2^+ (v)
&= B_2^+ (v)\ts  \R (u+v+C )\ts B^-_1 (u)\ts  \R (u-v -C),  \label{reflect3}
\end{align}
and the {\em unitarity relations}
\beq\label{reflect4}
B^\pm (u)\ts B^\pm (-u)=1.
\eeq
The reflection relations can be verified by using the defining relations \eqref{rtt12} and \eqref{rtt3} for the double Yangian and the identity 
$
R (u)\ts G_1\ts R (v)\ts G_2
=
 G_2\ts R (v)\ts G_1\ts R (u)$.

Write the matrices $B^\pm(u)$ in the form
$$
B^\pm(u) =\sum_{i,j=1}^{M+N} (-1)^{\bar{i}\bar{j}+\bar{j}}e_{ij}\ot b_{ij}^\pm(u),
$$
where the coefficients of the series $b_{ij}^\pm (u)$ are given by
$$
b_{ij}^-(u)=g_{ij}+ \sum_{r\geqslant 1}b_{ij}^{(r)}u^{-r}
\Fand 
b_{ij}^+ (u)=g_{ij}- \sum_{r\geqslant 1} b_{ij}^{(-r)} u^{r-1}.
$$
Clearly, all elements   $b_{ij}^{(r)}$ (resp. $b_{ij}^{(-r)}$) with $r=1,2,\ldots$ belong  to $\Y[C]$ (resp.  $\Ydext$). Furthermore, the elements $b_{ij}^{(\pm r)}$ are homogenous and their parity   is  $\bar{i}+\bar{j}$. 
Denote by $\B$ (resp. $\Bd$) the subalgebra of 
$\Y[C]$ (resp.   $\Ydext$) generated by all   $b_{ij}^{(r)}$ (resp. $b_{ij}^{(-r)}$).
Finally, let $\DB$ be the  subalgebra of  $\dygextfin$
generated by   elements $C$ and  $b_{ij}^{(\pm r)}$.
We   refer to  
these subalgebras as {\em reflection algebras}.

Recall the isomorphism given by \eqref{isomorphism}. 
Denote by $\bar{b}_{ij}^{(\pm r)}$  (resp. $\bar{C}$) the image of the element $b_{ij}^{(\pm r)}$ (resp. $C$) in the corresponding component of the graded algebra 
 $\gr \DB\subset \gr \dygextfin\cong \U(\glmnht) $. 
Introduce the families $I_0,I_1\subset \left\{1,\ldots ,M+N\right\}^{\times 2}$    by
\begin{align}
&I_0=\left\{(i,j)\,:\,1\leqslant i,j\leqslant\ell\text{ or }\ell +1\leqslant i,j\leqslant M+N\right\},\label{indices0}\\
&I_1=\left\{(i,j)\,:\,1\leqslant i \leqslant\ell <j\leqslant M+N\text{ or }1\leqslant j \leqslant\ell <i\leqslant M+N\right\}.\label{indices1}
\end{align}
Next, introduce the sets
\begin{align}
&\Gamma^- = \left\{\bar{b}_{ij}^{( 2r )}\,:\, (i,j)\in I_1\text{ and }r\geqslant 1\right\}\cup \left\{\bar{b}_{ij}^{(2r-1)}\,:\, (i,j)\in I_0\text{ and }r\geqslant 1\right\} 
 ,\label{gammam}\\
&\Gamma^+=\left\{\bar{b}_{ij}^{(-2r)}\,:\, (i,j)\in I_0\text{ and }r\geqslant 1\right\}\cup
\left\{\bar{b}_{ij}^{(-2r+1)}\,:\, (i,j)\in I_1\text{ and }r\geqslant 1\right\}.\label{gammap} 
\end{align}
Let $\Gamma=\Gamma^-\cup\Gamma^+$
and $\Gamma_C=\Gamma\cup\left\{\bar{C}\right\}$.
Clearly, we have
$
\Gamma^- \subset  \gr \B$ and $\Gamma^+ \subset  \gr \Bd$.
The next theorem can be verified by repeating the  proof of \cite[Thm. 3.1]{BK}.

\begin{thm}\label{thm42}
The set
$\Gamma^- $ (resp. $\Gamma^+$) generates   $\gr  \B$ (resp. $\gr  \Bd$). Therefore, the set $\Gamma_C$ generates the algebra $\gr  \DB$.
\end{thm}

Let $\sigma $
be the involutive automorphism   of $\mathfrak{gl}_{M|N}  $ given by
\begin{align*}
\sigma\colon e_{ij} \mapsto \varepsilon_i\ts \varepsilon_j \ts e_{ij},\quad\text{where } i,j=1,\ldots ,M+N.
\end{align*}
It induces the decomposition $\mathfrak{gl}_{M|N}=\mathfrak{gl}_{M|N} (-1) \oplus\mathfrak{gl}_{M|N} (1) $, where $\mathfrak{gl}_{M|N} (\pm 1) $ is the eigenspace of $\sigma$ corresponding to the eigenvalue $\pm 1$. Denote by $\mathfrak{gl}_{M|N}[t,t^{-1}]^{\sigma} $   the Lie  superalgebra of all Laurent   polynomials  
$$
 \sum_{i=-p}^q a_i \ot t^i \in  \glmn\ot\CC[t,t^{-1}] ,\quad\text{such that }p,q\in\ZZ_{\geqslant 0}\text{ and }a_{i}\in \mathfrak{gl}_{M|N}((-1)^i) .
$$
As with \cite[Cor. 3.2]{BK}, by using  the isomorphism given by \eqref{isomorphism} and Theorem \ref{thm42}, one easily deduces the following corollary.

\begin{kor}\label{reflection_cor}
The   assignments  
$$
\bar{b}_{ij}^{(r)}\mapsto (-1)^{\bar{i}}\left((-1)^{r-1}\varepsilon_i +\varepsilon_j\right)e_{ij}(r-1) \fand
\bar{b}_{ij}^{(-r)}\mapsto(-1)^{\bar{i}}((-1)^r\varepsilon_i +\varepsilon_j)e_{ij} (-r)  
$$
for $\bar{b}_{ij}^{(\pm r)}\in \Gamma $, together with  $\bar{C}\mapsto K$, define   an isomorphism of $\mathbb{Z}_2$-graded algebras
$$
\textstyle\gr  \DB\cong \U(\mathfrak{gl}_{M|N}[t,t^{-1}]^{\sigma}\oplus \CC K ).
$$
\end{kor}

\subsection{Presentation of reflection algebras} \label{subsection_32}

Let $G=(g_{ij})\in\ndo\CC^{M|N}$ be the   matrix as in \eqref{gmatricca}.
Define $\DBB$ as the $\ZZ_2$-graded  unital associative algebra generated by the elements $\mathcal{C}$ and  $\beta_{ij}^{(\pm r)}$, where $i,j=1,\ldots ,M+N$ and $r=1,2,\ldots .$   The parity of   $\beta_{ij}^{(\pm r)}$ is $\bar{i}+\bar{j}$ and the element $\mathcal{C}$ is even and central. The generators are  subject to the   defining relations which are written in terms of the matrices
$$
\mathcal{B}^{\pm }(u)=\sum_{i,j=1}^{M+N} (-1)^{\bar{i}\bar{j}+\bar{j}} e_{ij}\ot \beta^{\pm }_{ij}(u),
$$
where
$$
\beta^-_{ij}(u)=g_{ij}+\sum_{r\geqslant 1} \beta_{ij}^{(r)} u^{-r}
\Fand
\beta^+_{ij}(u)=g_{ij}-\sum_{r\geqslant 1} \beta_{ij}^{(-r)} u^{r-1}.
$$
  The defining relations consist of  three {\em reflection relations},
\begin{align}
R (u-v)\ts \mathcal{B}_1^\pm (u)\ts  R (u+v )\ts \mathcal{B}_2^\pm (v)
&= \mathcal{B}^\pm_2(v)\ts  R (u+v )\ts \mathcal{B}^\pm_1(u)\ts R (u-v), \label{reflect12gen} \\
 R (u-v +\mathcal{C})\ts \mathcal{B}^-_1 (u)\ts  R (u+v -\mathcal{C})\ts \mathcal{B}_2^+ (v)
&= \mathcal{B}_2^+ (v)\ts  R (u+v+\mathcal{C} )\ts \mathcal{B}^-_1 (u)\ts  R (u-v -\mathcal{C}),  \label{reflect3gen}
\end{align}
and two {\em unitarity relations},
\beq\label{reflect4gen}
\mathcal{B}^\pm (u)\ts \mathcal{B}^\pm (-u)=1.
\eeq

Clearly, the relations above are of the same form as   \eqref{reflect12}--\eqref{reflect4}, 
so the assignments
$$
\mathcal{C}\mapsto C\Fand
\beta_{ij}^{(\pm r)}\mapsto b_{ij}^{(\pm r)},\quad\text{where }i,j=1,\ldots ,M+N \text{ and } r\geqslant 1,$$
define the $\ZZ_2$-graded  algebra epimorphism   
\beq\label{bbmapY}
\DBB\to  \DB.
\eeq
Consider the ascending filtration over $\DBB$ defined by the degree operator 
$$
\deg\mathcal{C}=0,\quad
\deg   \beta_{ij}^{(r)}=r-1  \fand  \deg  \beta_{ij}^{(-r)}=-r, 
$$
where $i,j=1,\ldots ,M+N$ and $r=1,2,\ldots .$ 
The map   \eqref{bbmapY} is filtration-preserving and it gives rise to the epimorphism of the corresponding graded algebras
\beq\label{grbbmapY}
 \gr  \DBB\to  \gr  \DB \cong \U(\mathfrak{gl}_{M|N}[t,t^{-1}]^{\sigma}\oplus \CC K ).
\eeq  
As before, we write $\bar{\beta}_{ij}^{(r)}$  (resp. $\bar{\mathcal{C}}$) for the image  of the generator  $ \beta_{ij}^{(r)}$  (resp. $ \mathcal{C} $)
in the respective component of the corresponding graded algebra $\gr   \DBB$.

Let $\BB$ (resp. $\BBd$) be a unital subalgebra of $\DBB$ generated by   all $\beta_{ij}^{(r)}$   (resp.  $\beta_{ij}^{(-r)}$), where $i,j=1,\ldots ,M+N$ and $r=1,2,\ldots.$
Recall the sets of indices $I_0$ and $I_1$ defined  by \eqref{indices0} and \eqref{indices1}.
 Introduce the sets
\begin{align*}
&\mathcal{G}^- =\left\{\bar{\beta}_{ij}^{( 2r )}\,:\, (i,j)\in I_1\text{ and }r\geqslant 1\right\}\cup \left\{\bar{\beta}_{ij}^{(2r-1)}\,:\, (i,j)\in I_0\text{ and }r\geqslant 1\right\}  
 ,\\
&\mathcal{G}^+=\left\{\bar{\beta}_{ij}^{(-2r)}\,:\, (i,j)\in I_0\text{ and }r\geqslant 1\right\}\cup
\left\{\bar{\beta}_{ij}^{(-2r+1)}\,:\, (i,j)\in I_1\text{ and }r\geqslant 1\right\}  .
\end{align*}
Clearly, these are analogues of the sets $\Gamma^-$ and $\Gamma^+$, given by 
\eqref{gammam} and \eqref{gammap}, and they satisfy
$\mathcal{G}^-\subset\gr \BB$ and  $\mathcal{G}^+\subset\gr \BBd$. 
Let $\mathcal{G} =\mathcal{G}^-\cup\mathcal{G}^+$ and $\mathcal{G}_{\mathcal{C}} =\mathcal{G} \cup\left\{\bar{\mathcal{C}}\right\}$.
 
\begin{lem}\label{lemabbplus}
 Fix some ordering on the set  $\mathcal{G}$. 
Then any element of $\gr  \BB$ (resp. $\gr  \BBd$) can be written   as a linear combination of ordered monomials in elements of $ \mathcal{G}^- $
(resp. $ \mathcal{G}^+$) with at most power $1$ for odd generators.
\end{lem} 

\begin{prf}
Consider  the reflection algebra $\DB$ {\em at the level zero}, i.e., the quotient of $\DB$ over its two-sided ideal generated by the element $\mathcal{C}$.
At the level zero, this lemma coincides with \cite[Lemma 3.6]{BK}.
The lemma can be verified by repeating the proof of \cite[Lemma 3.6]{BK} verbatim. Indeed, this is due to the fact that the proof employs only  the reflection relations  \eqref{reflect12gen}  and the unitarity relations \eqref{reflect4gen}, so that it does not  depend on the element  $\mathcal{C}$.
\end{prf}

The proof of the next lemma  resembles   the proof of \cite[Lemma 3.7]{BK}.
Nonetheless, we present it in full detail, as it uses
 the remaining  reflection equation \eqref{reflect3gen}, which is slightly more complicated than its level zero counterpart employed in \cite[Lemma 3.7]{BK}.

\begin{lem}\label{lemauredjaj}
 Fix some ordering on the set $ \mathcal{G}_{\mathcal{C}}  $ such that all elements of $ \mathcal{G}^+$ precede all elements of $ \mathcal{G}^-$. Then any element of $\gr   \DBB$ can be written as a linear combination of ordered monomials in $ \mathcal{G}_{\mathcal{C}}  $ with at most power 1 for odd generators.
\end{lem}

\begin{prf}
It  suffices to prove that any element  $z\in\gr_2 \DBB$ can be written as a linear combination of   elements of the form $w  x  y$, where $w\in\CC[\bar{\mathcal{C}}]$, $x\in\gr_2\BBd$ and $y\in\gr_2\BB$. Indeed, the      lemma then   follows by applying Lemma \ref{lemabbplus} to $x$ and $y$. 
Moreover,   Lemma  \ref{lemabbplus} implies that the set $ \mathcal{G}_{\mathcal{C}}$ generates $\gr_2  \DBB$, so we can assume without loss of generality that $z$ is a monomial in $ \mathcal{G}_{\mathcal{C}}$. 
For simplicity, let us also assume that $\bar{\mathcal{C}}$ is the smallest element with respect to the given ordering.
We shall now   describe a procedure  which one can use to express $z$  in such a way.

Let
$$
G^\pm=\left\{\beta_{ij}^{(r)}\,:\, \bar{\beta}_{ij}^{(r)}\in\mathcal{G}^\pm\right\},\quad
G=G^- \cup G^+\fand
G_{\mathcal{C}}=G\cup\left\{\mathcal{C}\right\}.
$$
 Suppose  $\mu$ is a monomial in $ G_{\mathcal{C}}$ of the form
$$
\mu =\mathcal{C}^d  \mu_1\ts  \beta_{ij}^{(r)}\ts\beta_{kl}^{(-s)}\ts\mu_2,
$$
where $d\geqslant 0$, $i,j,k,l=1,\ldots ,M+N$, $r,s=1,2,\ldots$ and
  $\mu_1,\mu_2$ are some monomials in $ G $. As $\bar{\beta}_{kl}^{(-s)}$ precedes $\bar{\beta}_{ij}^{(r)}$ with respect to the chosen ordering, the monomial $\bar{\mu}$ is not ordered.
	
	Consider the matrix $\R(u)$ defined by $\eqref{Rbar}$. By \eqref{R} and \eqref{function_g}, we see that
	$$
\R(u)	=\left(1+\sum_{a\geqslant 1} \frac{g_a}{u^a}\right) - P\left(\frac{1}{u} +\sum_{a\geqslant 1} \frac{g_a}{u^{a+1}}\right)
=1+\sum_{a\geqslant 1}  R_a \ts u^{-a},
	$$
	where $R_a\in\ndo\CC^{M|N}\ot\ndo\CC^{M|N}$ are given by
	$$
	R_1=g_1\ts I- P
	\Fand
R_a =g_a\ts I -g_{a-1}\ts P	\quad\text{for}\quad a\geqslant 2 .
	$$
Hence, the  $R$-matrices
 which appear in the reflection equation \eqref{reflect3gen}, are of  the form
	$$
	 \R(u+ \nu_1 v+\nu_2\ts \mathcal{C} )=
	1+
	\sum_{a\geqslant 1}\frac{R_a}{u^a}\left(\sum_{b\geqslant 0}(-1)^b\frac{\left(\nu_1 v+\nu_2\ts \mathcal{C}\right)^b}{u^b}\right)^a
	\quad\text{for }\nu_1,\nu_2\in\left\{\pm 1\right\}.
	$$	
For all $i,j,k,l=1,\ldots ,M+N$ define
		$$
	\beta_{ij}^{kl}(u,v)=\left(\beta_{ij}(u)-g_{ij}\right)\left(g_{kl}-\beta_{kl}^+(v)\right).
	$$
 The degree of the coefficient of $u^{-r}v^{s-1}$, where $r,s\geqslant 2$, in   $\frac{v^m}{u^m}\beta_{ij}^{kl}(u,v)$
	(resp. $\frac{v^m}{u^n}\beta_{ij}^{kl}(u,v)$ with $n>m$)
	is less than or equal to $r-s-1$ (resp. less than or equal to $r-s-2$).  
By using these observations and extracting the coefficients of $u^{-r}v^{s-1}$ of the matrix entry $e_{ij}\ot e_{kl}$ in the reflection equation \eqref{reflect3gen}, we get
\beq\label{imgaeof2}
\beta_{ij}^{(r)}\ts\beta_{kl}^{(-s)}+\gamma_1+\gamma_2
=
\pm\beta_{kl}^{(-s)}\ts\beta_{ij}^{(r)}+\delta_1+\delta_2,
\eeq
where the sign $\pm$ on the right-hand side depends on the parity of the given generators, $\gamma_1,\delta_1$ are  linear combinations of the elements $\beta_{pq}^{(\pm t)}$ of degree less than or equal to $r-s-1$ and $\gamma_2,\delta_2$ are  linear combinations of some  monomials  in $ G_{\mathcal{C}}$ of degree less than or equal to $r-s-2$. Hence, by taking the image of \eqref{imgaeof2} in the $(r-s-1)$-component of the corresponding graded algebra, we express $\bar{\mu}$ as a linear combination
\beq\label{nuredjenimnomii}
\bar{\mu}=\pm \ts \bar{\mathcal{C}}^d\ts \bar{\mu}_1\ts \bar{\beta}_{kl}^{(-s)}\ts\bar{\beta}_{ij}^{(r)}\ts \bar{\mu}_2
+\ts \bar{\mathcal{C}}^d\ts\bar{\mu}_1\ts \bar{\delta}_1\ts \bar{\mu}_2
-\ts \bar{\mathcal{C}}^d\ts\bar{\mu}_1\ts \bar{\gamma_1}\ts \bar{\mu}_2.
\eeq
Observe that the lengths of the monomials which appear in the linear combinations  $  \bar{\mathcal{C}}^d\ts\bar{\mu}_1\ts \bar{\delta}_1\ts \bar{\mu}_2$ and $ \bar{\mathcal{C}}^d\ts\bar{\mu}_1\ts \bar{\gamma_1}\ts \bar{\mu}_2$
are strictly less than the length of $\bar{\mu}$.
Therefore, we can continue to apply such a procedure, now
starting with the monomials on the right-hand side of \eqref{nuredjenimnomii}, until, after finitely many
steps, we obtain a linear combination of   the elements   of the form $wxy$ with $w\in\CC[\bar{\mathcal{C}}]$, $x\in\gr_2\BBd$ and $y\in\gr_2\BB$, as required.
\end{prf}

Finally, the following theorem gives a presentation of reflection algebras. It can be proved by a usual argument which relies on Corollary \ref{reflection_cor}, Lemma \ref{lemauredjaj} and 
 the Poincar\'{e}--Birkhoff--Witt Theorem for the universal enveloping algebra $ \U(\mathfrak{gl}_{M|N}[t,t^{-1}]^{\sigma} \oplus\CC K)$; for more details see the proof of \cite[Thm. 3.8]{BK}.

\begin{thm}
The map \eqref{bbmapY} is a $\ZZ_2$-graded  algebra isomorphism.
\end{thm}

\section{Vacuum module as an \texorpdfstring{$h$}{h}-adic quantum vertex superalgebra}\label{section_4}
In this section, we equip the vacuum module  $V^c(\glmn)$ over the double Yangian by the structure of $h$-adic quantum vertex superalgebra and we establish a partial connection between $V^c(\glmn)$-modules and restricted $\dyg$-modules of level $c\in\CC$.

From now on, we consider the double Yangian $\dyg$ over the commutative ring $\CC[[h]]$ of all formal power series in parameter $h$ with complex coefficients.
The definition of the double Yangian from Section \ref{section_2} is translated to the $h$-adic setting by formal rescaling $u\mapsto u/h$ of the variables and generators
$$
t_{ij}^{(r)}\mapsto h^{r-1}\ts t_{ij}^{(r)},\qquad
t_{ij}^{(-r)}\mapsto h^{-r }\ts t_{ij}^{(-r)}\Fand
C\mapsto C
$$
for all $r=1,2,\ldots .$ The $R$-matrices  \eqref{R} and \eqref{Rbar} are now given by
\beq\label{hadicrmatrix}
R(u)=I-Phu^{-1}\Fand 
\R(u)=g(u/h)\left(I-Phu^{-1}\right) 
\eeq
and they exhibit the same properties as their complex field counterparts  from Section \ref{section_2}. In particular, the crossing symmetry properties \eqref{csym} are given by
\beq\label{hadic_csym}
\left(\R(u-(M-N)h)^{-1}\right)^{\tau_i}\ts \R(u )^{\tau_i}=I\qquad\text{for }i=1,2.
\eeq

The
{\em double Yangian $\dyg$ for $\mathfrak{gl}_{M|N}$} is defined as the $\ZZ_2$-graded unital associative algebra over $\CC[[h]]$ generated by the elements $C$ and $t_{ij}^{(\pm r)}$  with $  i,j=1,\ldots , M+N$ and $r=1,2,\ldots,$ 
where $C$ is   even and central   and  the parity of
  $t_{ij}^{(\pm r)}$  is  $\bar{i}+\bar{j}$.
The generators are
subject to the  defining relations which are written in terms of the  matrices 
\beq\label{hadic_matrices}
T^\pm (u)=\sum_{i,j=1}^{M+N} (-1)^{\bar{i}\bar{j}+\bar{j}} e_{ij}\ot t^\pm_{ij}(u),
\eeq
 where
$$
t_{ij}^-(u)=\delta_{ij}+h\sum_{r\geqslant 1} t_{ij}^{(r)} u^{-r}
\Fand 
t_{ij}^+ (u)=\delta_{ij}-h\sum_{r\geqslant 1} t_{ij}^{(-r)} u^{r-1}.
$$
The relations  are given by  
\begin{align}
R_{12}(u-v)\ts T_1^\pm (u)\ts T_2^\pm (v)
&= T^\pm_2(v)\ts T^\pm_1(u)\ts R_{12}(u-v), \label{hadic_rtt12}\\
\R_{12}(u-v+hC/2)\ts T^-_1 (u)\ts T_2^+ (v)
&= T_2^+ (v)\ts T^-_1 (u)\ts \R_{12}(u-v-hC/2).\label{hadic_rtt3} 
\end{align}
 We assume that the double Yangian is completed with respect to the $h$-adic topology.

As in Section \ref{section_2},
the {\em Yangian $\Y$} (resp. the {\em dual Yangian $\Yd$}) is  defined as the subalgebra of the double Yangian   generated by all elements $t_{ij}^{(r)}$ (resp. $t_{ij}^{(-r)}$) with $r=1,2,\ldots .$
Finally, the {\em vacuum module $V^c(\glmn)$ of level $c\in\CC$} is defined as the quotient of the double Yangian $\dyg$ by its $h$-adically completed left ideal generated by $C-c\cdot 1$ and all elements $t_{ij}^{(r)}$ with   $r=1,2,\ldots .$ As a $\CC[[h]]$-module,   $V^c(\glmn)$ is isomorphic to the $h$-adic  completion of the dual Yangian $\Yd$. Thus, the dual Yangian induces the $\ZZ_2$-grading on the vacuum module, 
\beq\label{ztwograding}
V^c(\glmn) = V^c(\glmn)_{\bar{0}}\oplus V^c(\glmn)_{\bar{1}}  .
\eeq

Suppose $W$ is a $\CC[[h]]$-module. Denote by  $W((z))_h$ (resp. $W[z]_h$)  the $h$-adic completion of $W((z)) $ (resp. $W[z]$). Hence, if $W$ is topologically free, i.e. $W=W^\prime[[h]]$ for some complex vector space $W^\prime$, we have $W((z))_h=W^\prime ((z))[[h]]$ and $W[z]_h= W^\prime[z][[h]]$.
For more information on the $h$-adic topology and topologically free modules see the book by Kassel \cite{Kas}.

Let $W$ be
 a module over the superalgebra  $\dyg$. Then $W$ is said to be a {\em restricted module}, if 
it is topologically free as a $\CC[[h]]$-module and  
the action $T^-(u)_W$ of the matrix $T^-(u) $ belongs to $\ndo\CC^{M|N}\ot \om(W,W[u^{-1}]_h)$. 
Furthermore, if the canonical central element  $C\in\dyg$ acts on $W$ as a   scalar multiplication by $c\in\CC$, then $W$ is said to be of {\em level $c$}.
For example, the vacuum module $V^c(\glmn)$ is a restricted 
 $\dyg$-module of level $c$. 

Let $\vac $ be the image of the unit in the double Yangian under the canonical projection  $\dyg \to V^c(\glmn)$. From now on, we consider the matrices $T^\pm(u)$ as operators on the vacuum module, i.e., as elements of $\ndo\CC^{M|N}\ot\ndo V^c(\glmn)[[u^{\pm 1}]]$.
To declutter the notation, for the   variables $v=(v_1,\ldots ,v_n)$ and the single variable $z$, we write
$$
T_{[n]}^{\pm}(v)=T_1^{\pm}(v_1)\ldots T_n^{\pm}(v_n)\Fand
T_{[n]}^{\pm}( v|z)=T_1^{\pm}(z+v_1)\ldots T_n^{\pm}(z+v_n).
$$
Note that the coefficients of the above expressions with respect to the variables $v_1,\ldots ,v_n$ and $z$ belong to
$(\ndo\CC^{M|N})^{\ot n}\ot\ndo V^c(\glmn)$.
By using the defining relation \eqref{hadic_rtt3}, we find that the action of $T^-(u)$ on the vacuum module is given by
$$
T^-_0(u)\ts T_{[n]}^{+}(v)\vac
=\R_{01}^{- }\ldots \R_{0n}^{- }\ts 
T_{[n]}^{+}(v)\vac\ts 
\R_{0n}^{+}\ldots \R_{01}^{+},
$$
where
$
\R_{0i}^{- }=\R_{0i}(u-v_i+hc/2)^{-1 }$ and
$\R_{0i}^{+}=\R_{0i}(u-v_i-hc/2)$.

From now on, the tensor products are understood as $h$-adically completed.
The following definition is a  supersymmetric counterpart of the
notion of $h$-adic quantum vertex algebra \cite[Def. 2.20]{Li} which is in turn a slight generalization of the notion of quantum VOA \cite[Subsect. 1.4.1]{EK}.

\begin{defn}\label{qvoa}
An {\em $h$-adic quantum vertex superalgebra} is a quadruple $(V,Y,\vac,\Sc)$ as follows.
\begin{enumerate} 
\item  $V=V_{\bar{0}}\oplus V_{\bar{1}}$ is a $\ZZ_2$-graded topologically free $\mathbb{C}[[h]]$-module.
\item $Y=Y(\cdot, z)$ is the {\em vertex operator map}, a $\mathbb{C}[[h]]$-module map
\begin{align*}
Y(\cdot, z) \colon V &\to \om(V,V((z))_h)\\
u &\mapsto  Y(u,z) =\sum_{r\in\mathbb{Z}} u_r \ts z^{-r-1}
\end{align*}
such that
\beq\label{super_01}
Y(u,z)v\in V_{i+j}[[z^{\pm 1}]]\quad\text{for all }i,j\in\ZZ_2,\,u\in V_i,\, v\in V_j, 
\eeq
which satisfies the {\em weak associativity}:
for any $u,v,w\in V$ and $p\in\mathbb{Z}_{\geqslant 0}$
there exists $q\in\mathbb{Z}_{\geqslant 0}$
such that
$$
(z_0 +z_2)^q\ts Y(u,z_0 +z_2)Y(v,z_2)\ts w - (z_0 +z_2)^q\ts Y\big(Y(u,z_0)v,z_2\big)\ts w
\in h^p V[[z_0^{\pm 1},z_2^{\pm 1}]].
$$
\item $\vac$ is the {\em vacuum vector}, a distinct element of $V_{\bar{0}}$  such that
$$
Y(\vac ,z)v=v\quad\text{for all }v\in V,
$$
which satisfies 
$$
Y(v,z) \vac\in V[[z]]\fand
\lim_{z\to 0} Y(v,z)\ts\vac =v\qquad\text{for all }v\in V.
$$
\item $\Sc=\Sc(z)$ is the {\em braiding}, a $\mathbb{C}[[h]]$-module map
$$\Sc(z)\colon V\otimes V\to V\otimes V\otimes\mathbb{C}((z))[[h]]$$
such that
\begin{align}
&\Sc(z)\left((V\ot V)_{i}\right)\subset  (V\ot V)_{i} \otimes\mathbb{C}((z))[[h]]\quad\text{for all }i\in\ZZ_2,\label{super_02}\\
\intertext{where
$(V\ot V)_{i}=\oplus_{a+b=i} V_a\ot V_b $,
which satisfies the {\em shift condition}
}
&[\Dc\otimes 1, \mathcal{S}(z)]=-\frac{d}{dz}\mathcal{S}(z)\quad \text{for}\quad \Dc\in\ndo V\text{   defined by }\Dc v=v_{-2}\vac,\non\\
\intertext{the {\em quantum Yang--Baxter equation}}
&\mathcal{S}_{12}(z_1)\ts\mathcal{S}_{13}(z_1+z_2)\ts\mathcal{S}_{23}(z_2)
=\mathcal{S}_{23}(z_2)\ts\mathcal{S}_{13}(z_1+z_2)\ts\mathcal{S}_{12}(z_1),\non\\
\intertext{the {\em unitarity condition}}
&\mathcal{S}(z)^{-1}=\mathcal{S}_{21}(-z),\non
\end{align}
the {\em $\mathcal{S}$-locality}:
for any $i,j\in\ZZ_2$, $u \in V_i$, $v\in V_j$ and $p\in\mathbb{Z}_{\geqslant 0}$ there exists
$q\in\mathbb{Z}_{\geqslant 0}$ such that for all $w\in V$ we have
\begin{align}
&(z_1-z_2)^{q}\ts Y(z_1)\big(1\otimes Y(z_2)\big)\big(\mathcal{S}(z_1 -z_2)(u\otimes v)\otimes w\big)
\nonumber\\
&\quad-(-1)^{ij}(z_1-z_2)^{q}\ts Y(z_2)\big(1\otimes Y(z_1)\big)(v\otimes u\otimes w)
\in h^p V[[z_1^{\pm 1},z_2^{\pm 1}]] ,\label{locality}
\intertext{and the {\em hexagon identity}}
&\Sc(z_1)\left(Y(z_2)\ot 1\right) =\left(Y(z_2)\ot 1\right)\Sc_{23}(z_1)\Sc_{13}(z_1+z_2).\non
\end{align}
\end{enumerate}
\end{defn}

Next, we consider the notion of module over $h$-adic quantum vertex superalgebra, which is a   supersymmetric counterpart of   \cite[Def. 2.23]{Li}.

\begin{defn}\label{qvoamodule}
 Let $(V,Y,\vac ,\Sc)$  be an $h$-adic quantum vertex superalgebra. A {\em   $V$-module} is a pair $(W,Y_W)$, where $W=W_{\bar{0}}\oplus W_{\bar{1}}$ is a  $\ZZ_2$-graded topologically free $\CC[[h]]$-module
and 
\begin{align*}
Y_W = Y_W(\cdot ,z)\colon V &\to\om(W, W((z))_h)\\
u &\mapsto Y_W(u,z) =\sum_{r\in\mathbb{Z}} u_r  \ts z^{-r-1}
\end{align*}
is a $\CC[[h]]$-module map such that
\beq\label{super_03}
Y_W(u,z)w\in W_{i+j}[[z^{\pm 1}]]\quad\text{for all }i,j\in\ZZ_2,\,u\in V_i,\, w\in W_j ,
\eeq
which satisfies  
 $$Y_W(\vac,z)w=w\quad\text{for all } w\in W$$
 and
the {\em weak associativity}:
for any $u,v\in V$, $w\in W$ and $p\in\mathbb{Z}_{\geqslant 0}$
there exists $q\in\mathbb{Z}_{\geqslant 0}$
such that
\begin{align*}
&(z_0 +z_2)^q\ts Y_W (u,z_0 +z_2)Y_W(v,z_2)\ts w   - (z_0 +z_2)^q\ts Y_W\big(Y(u,z_0)v,z_2\big)\ts w
\in h^p W[[z_0^{\pm 1},z_2^{\pm 1}]]. 
\end{align*}

We define the notion of {\em quasi $V$-module} by replacing the weak associativity with the following requirement: for any $u,v\in V$, $w\in W$ and $p\in\mathbb{Z}_{\geqslant 0}$
there exists a nonzero polynomial  $q(z_1,z_2)\in\CC[z_1,z_2]$
such that
\begin{align*}
&q(z_0 +z_2,z_2) \ts Y_W (u,z_0 +z_2)Y_W(v,z_2)\ts w   - q(z_0 +z_2,z_2) \ts Y_W\big(Y(u,z_0)v,z_2\big)\ts w
\end{align*}
belongs to  $h^p W[[z_0^{\pm 1},z_2^{\pm 1}]]$. 
\end{defn}

In the next theorem, we apply the Etingof--Kazhdan construction \cite[Thm. 2.3]{EK} to the vacuum module.  
To present the explicit expression for the braiding, we shall need the following notation.
Let   $u=(u_1,\ldots ,u_n)$, $v=(v_1,\ldots, v_m)$ be   families of variables and $z$ a single variable.
Denote by the superscripts $1,2,3,4$ the   tensor factors as follows:
\beq\label{nhk78}
\overbrace{(\ndo\mathbb{C}^{M|N})^{\ot  n}}^{1} \ot 
\overbrace{(\ndo\mathbb{C}^{M|N})^{\ot  m}}^{2}\ot
\overbrace{V^c(\glmn)}^{3} \ot \overbrace{V^c(\glmn)}^{4}.
\eeq
Define the  formal power series   with coefficients in
$ (\ndo\mathbb{C}^{M|N})^{\ot  n}   \ot 
 (\ndo\mathbb{C}^{M|N})^{\ot  m} $,
\begin{align}
&\R_{nm}^{12}(u|v|z+ah)= \prod_{i=1,\dots,n}^{\longrightarrow}
\prod_{j= 1,\ldots, m}^{\longleftarrow} \R_{i\ts n+j}(z+u_i -v_{j}+ah) ,\label{RnM1}\\
&\wndr{\R}_{nm}^{12}(u|v|z+ah)= \prod_{i=1,\dots,n}^{\longrightarrow}
\prod_{j= 1,\ldots, m}^{\longrightarrow} \R_{i\ts n+j}(z+u_i +v_{j}+ah)\label{RnM2} ,
\end{align}
  where the arrows indicate the order of   factors and $a\in\CC$.  For example, 
	by setting $n=m=2$ and $a=0$ 
	in \eqref{RnM1} and \eqref{RnM2}, we get  
	$$
	\R_{nm}^{12}(u|v|z)=\R_{14}\ts\R_{13}\ts\R_{24}\ts\R_{23}\Fand
	\wndr{\R}_{nm}^{12}(u|v|z)=\wndr{\R}_{13}\ts\wndr{\R}_{14}\ts\wndr{\R}_{23}\ts\wndr{\R}_{24},
	$$
		where $\R_{ij}=\R_{z+u_i-v_j}$ and $\wndr{\R}_{ij}=\R_{z+u_i+v_j}$.
Furthermore, we write
	\begin{align}
&\R_{nm}^{12}(u+ah|v )= \prod_{i=1,\dots,n}^{\longrightarrow}
\prod_{j= 1,\ldots, m}^{\longleftarrow} \R_{i\ts n+j}( u_i -v_{j}+ah) , \label{RnM3}\\
&\wndr{\R}_{nm}^{12}(u+ah|v )= \prod_{i=1,\dots,n}^{\longrightarrow}
\prod_{j= 1,\ldots, m}^{\longrightarrow} \R_{i\ts n+j}( u_i +v_{j}+ah)\label{RnM4}
\end{align}	
if the variable $z$ is omitted in in \eqref{RnM1} and \eqref{RnM2}.
		Throughout the paper, we shall apply the same notation conventions to the $R$-matrix $R(u)$ as well.

\begin{thm}\label{glavni_teorem}
\begin{enumerate}
\item
For any $c\in\CC$ there exists a unique structure of $h$-adic quantum vertex superalgebra on the vacuum module $V^c(\glmn) = V^c(\glmn)_{\bar{0}}\oplus V^c(\glmn)_{\bar{1}} $ such that  the vertex operator map is given by
\beq\label{hadic_y}
Y(T_{[n]}^+(u ) \vac ,z)=
T_{[n]}^+(u|z )\ts 
T_{[n]}^-(u|z+hc/2 )^{-1} ,
\eeq 
where $u=(u_1,\ldots ,u_n)$,
the vacuum vector is $\vac$ and the braiding   $\mathcal{S} $ satisfies the identities  on 
\eqref{nhk78} with $v=(v_1,\ldots ,v_m)$, 
\begin{align}
&\mathcal{S}^{34}(z)\Big( T_{[m]}^{+24}(v) \ts
\R	_{nm}^{  12}(u|v|z-h    c) \ts T_{[n]}^{+13}(u)\ts \R_{nm}^{  12}(u|v|z)^{-1}\ts(\vac\otimes \vac) \Big) \non\\
=& \,\R_{nm}^{  12}(u|v|z)  \ts T_{[n]}^{+13}(u)  \ts\R_{nm}^{  12}(u|v|z+h    c)^{-1} \ts
 T_{[m]}^{+24}(v) \ts(\vac\otimes \vac).\label{hadic_s}
\end{align}
\item
 Let $W$ be a restricted  module of level $c$ over the superalgebra $\dyg$. There exists a unique structure of $V^c(\glmn)$-module on $W$ such that  
$$
Y_W(T_{[n]}^+(u ) \vac ,z)=
T_{[n]}^+(u|z )_W\ts 
T_{[n]}^-(u|z+hc/2 )^{-1}_W.
$$
\end{enumerate}
\end{thm}

\begin{prf} (1) The even case of this assertion, i.e. the construction of  $h$-adic quantum vertex algebra structure on the vacuum module $V^c(\gl_M)$ over the double Yangian ${\rm DY}(\gl_M)$,   goes back to Etingof and Kazhdan \cite[Thm. 2.3]{EK}. It is proved by straightforward computations which employ the usual $R$-matrix techniques to   verify the requirements imposed by the definition of  $h$-adic quantum vertex algebra, as demonstrated by the proofs of   \cite[Thm. 2.3.8]{Gar} and \cite[Thm. 4.1]{JKMY}. These proofs directly extend to the super case, due to the same form of the vertex operator map and the braiding. Therefore, we shall only     verify the requirements  from Definition \ref{qvoa} which are related to   $\ZZ_2$-grading, i.e., which are not addressed in the aforementioned proofs. More specifically, by comparing the notion of $h$-adic quantum vertex algebra \cite[Def. 2.23]{Li} with its supersymmetric counterpart from Definition \ref{qvoa}, we see that 
$V^c(\glmn)$ is now required to be $\ZZ_2$-graded and the vertex operator map (resp. the braiding map) needs to satisfy the additional requirement \eqref{super_01}  (resp. \eqref{super_02}). Finally, the  $\Sc$-locality axiom \eqref{locality} contains the term $(-1)^{ij}$,    depending on the parity of the corresponding homogenous elements $u,v\in V^c(\glmn)$, which equals $1$ in the even case.

Recall  that the $\CC[[h]]$-module $V^c(\glmn)$ is $\ZZ_2$-graded by \eqref{ztwograding}. Let us prove that the vertex operator map \eqref{hadic_y} and the braiding \eqref{hadic_s} satisfy the grading restrictions \eqref{super_01}  and \eqref{super_02}. 
Let $U(0)=\left\{\vac\right\}$ and let $U(n)\subset  V^c(\glmn)$ for $n\geqslant 1$ be the set of all coefficients, with respect to the variables $u_1,\ldots ,u_n$, of all matrix entries $e_{i_1\ts j_1}\ot\ldots\ot e_{i_n\ts j_n}$ of  
$T_{[n]}^+(u_1,\ldots ,u_n)\vac $. 
Note that the elements of the set $U(n)$ are homogenous with respect to the $\ZZ_2$-grading \eqref{ztwograding}, as all coefficients of the matrix entry 
$e_{i_1\ts j_1}\ot\ldots\ot e_{i_n\ts j_n}$ of  
$T_{[n]}^+(u_1,\ldots ,u_n)\vac $ are of parity $(\bar{i}_1+\ldots +\bar{i}_n)+ (\bar{j}_1+\ldots +\bar{j}_n)$.
Hence, the set $U=\cup_{n=0}^\infty U(n)$ can be written as a disjoint union
\beq\label{hadicdisjoint}
U=U_{\bar{0}}\cup U_{\bar{1}}, 
\eeq
where $U_{\bar{0}} $ (resp. $  U_{\bar{1}}$) consists of all even (resp. odd) elements of $U$.
For any subset $V\subseteq V^c(\glmn)$, define
$$
[V] =\left\{v\in V^c(\glmn)\,:\, u=h^n v\text{ for some }u\in V,\, n\geqslant 0\right\}.
$$
Let $V$ be the $\CC[[h]]$-span of  $U $. The  $\CC[[h]]$-span of $[V]$ is an $h$-adically dense $\CC[[h]]$-submodule of $V^c(\glmn)$. Therefore, to prove the grading restrictions, it suffices to show that for all $i,j\in\ZZ_2$, $u\in U_i$ and $v\in U_j$, we have
\beq\label{hadicassertion1}
Y(u,z)v\in V^c(\glmn)_{i+j}[[z^{\pm 1}]]
\eeq
and 
\beq\label{hadicassertion2}
\Sc(z)(  u\ot v) \in (V^c(\glmn)\ot V^c(\glmn))_{i+j}\ot \CC((z))[[h]].
\eeq

Regarding \eqref{hadicassertion2}, as evident from \eqref{hadicrmatrix}, all coefficients of the $R$-matrix $\R(u)$   with respect to the variable $u$ are even elements of $\ndo\CC^{M|N}\ot \ndo\CC^{M|N}[[h]]$, so  the operation of multiplication by the $R$-matrix does not change the parity. Therefore, all coefficients
of the matrix entry
\beq\label{hentries1}
e_{i_1\ts j_1}\ot\ldots\ot e_{i_n\ts j_n}\ot e_{k_1\ts l_1}\ot \ldots \ot e_{k_m\ts l_m}
\eeq
on both sides of \eqref{hadic_s} are of the same parity
\beq\label{hentries2}
(\bar{i}_1+\ldots +\bar{i}_n)+ (\bar{j}_1+\ldots +\bar{j}_n)+
(\bar{k}_1+\ldots +\bar{k}_m)+ (\bar{l}_1+\ldots +\bar{l}_m),
\eeq
which implies \eqref{hadicassertion2}.

As for the grading restriction \eqref{hadicassertion1}, it is evident from \eqref{hadic_matrices}, that
all coefficients of the matrices $T^\pm(u)$ with respect to the variable $u$ are even elements of $\ndo\CC^{M|N}\ot\ndo V^c(\glmn)$. 
Moreover, by \eqref{hadic_y}, we have
 \beq\label{hentries3}
Y(T_{[n]}^+(u ) \vac ,z)T_{[m]}^+(v )\vac=
T_{[n]}^{+13}(u|z )\ts 
T_{[n]}^{-13}(u|z+hc/2 )^{-1}\ts T_{[m]}^{+23}(v )\vac.
\eeq
Hence,   all coefficients of the matrix entry \eqref{hentries1} on both sides of
  \eqref{hentries3} are of the same parity \eqref{hentries2}, as required.

It remains to prove the $\Sc$-locality property.
It suffices to show that \eqref{locality} holds for all $i,j\in\ZZ_2$, $u \in U_i$ and  $v\in U_j$, where $U_{\bar{0}}$ and  $U_{\bar{1}}$ are the sets of even and odd elements of $U$; recall \eqref{hadicdisjoint}. As in the proof of \cite[Thm. 4.1]{JKMY}, one compares the expressions
\beq\label{hadiceexp1}
(z_1-z_2)^q\ts Y(z_1)\left(1\ot Y(z_2)\right)\Sc(z_1 -z_2) \ts T_{[m]}^{+24}(v)\ts T_{[n]}^{+13}(u)(\vac\ot\vac)
\eeq
and
\beq\label{hadiceexp2}
(z_1-z_2)^q\ts  Y(z_2)\left(1\ot Y(z_1)\right)  T_{[m]}^{+23}(v)\ts T_{[n]}^{+14}(u)(\vac\ot\vac),
\eeq
where $u=(u_1,\ldots ,u_n)$ and $v=(v_1,\ldots ,v_m)$.
More specifically, one finds that for all $p,r_1,\ldots ,r_n,s_1,\ldots ,s_m\geqslant 1$ there exists $q\geqslant 1$ such that all coefficients of 
$$
u_1^{a_1}\ldots u_n^{a_n} \ts v_1^{b_1}\ldots v_m^{b_m},\quad a_\alpha=0,\ldots ,r_\alpha,\,   b_\beta=0,\ldots ,s_\beta,\, \alpha=1,\ldots ,n,\, \beta=1,\ldots ,m,
$$
 of the matrix entries \eqref{hentries1} in \eqref{hadiceexp1} and in \eqref{hadiceexp2}  coincide modulo $h^p$. Note that   \eqref{hadiceexp1} (resp. \eqref{hadiceexp2}) contains the term $T_{[m]}^{+24}(v)\ts T_{[n]}^{+13}(u)$ (resp. $T_{[m]}^{+23}(v)\ts T_{[n]}^{+14}(u)$).
Thus, by extracting the matrix entries of  \eqref{hentries1}  in
  \eqref{hadiceexp1} and  in \eqref{hadiceexp2}, the former expression
produces the additional sign $(-1)^{i\ts j}$, where
$$
i=(\bar{i}_1+\ldots +\bar{i}_n)+ (\bar{j}_1+\ldots +\bar{j}_n)\fand 
j=(\bar{k}_1+\ldots +\bar{k}_m)+ (\bar{l}_1+\ldots +\bar{l}_m).
$$
Finally, as $i$ (resp. $j$)
equals the parity of all coefficients of the matrix entry 
$e_{i_1\ts j_1}\ot\ldots\ot e_{i_n\ts j_n}$
(resp. $e_{k_1\ts l_1}\ot \ldots \ot e_{k_m\ts l_m}$)
in $T_{[n]}^+(u_1,\ldots ,u_n)$ (resp. in $T_{[m]}^+(v_1,\ldots ,v_m)$), we conclude that the $\Sc$-locality holds.

\noindent (2) The second assertion  follows again by arguing as in the proofs of
\cite[Thm. 2.3.8]{Gar} and \cite[Thm. 4.1]{JKMY}.
The grading restriction \eqref{super_03} on the module map is an immediate consequence of the grading restriction in the definition of the notion of module over an associative superalgebra.
\end{prf}

\section{Central elements in the completed double Yangian at the critical level}\label{section_3}

For any $c\in\CC$, denote by $\dygc$ the {\em double Yangian at the level $c$}, i.e.,   the quotient of the algebra $\dyg$ over the $h$-adically two-sided  completed ideal generated by $C-c $. Next, let $I_p$ with $p\geqslant 1$ be the $h$-adically  completed left ideal of $\dygc$ generated by all $t_{ij}^{(r)}$ with $r\geqslant p$.
Define the completion of $\dygc$ as the  inverse limit
$$
\dygctld = \lim_{\longleftarrow} \dygc  / I_p.
$$
Throughout this subsection, we consider the completed double Yangian at the {\em critical level} $c=N-M$, which we denote by $\dygtld = \dygNMtld$.
Our goal  is to construct families of central elements in $\dygtld$.

Let us recall some  properties of the supertrace   $\str\colon \ndo\CC^{M|N}\to \CC$,  
$$
\str\colon e_{ij}\mapsto \delta_{ij}(-1)^{\bar{i} }\quad\text{for all }i,j=1,\ldots ,M+N. 
$$ 
First, it possesses the supercyclic property, 
\beq\label{str_cyclic}
\str A\ts B =\str B \ts A (-1)^{\bar{A}\bar{B}}\quad\text{for all homogenous }A,B\in \ndo\CC^{M|N},
\eeq
where $\bar{X}$ equals $0$ (resp. $1$) if $X$ is even (resp. odd) element. 
Next,  it respects the supertransposition $\tau$, as defined by \eqref{supertransposition}, i.e., it satisfies
\beq\label{str_transposition}
\str A = \str A^\tau\quad\text{for all }A\in \ndo\CC^{M|N}.
\eeq
Also, we remark that  the supertransposition is an antihomomorphism, so that we have
\beq\label{tau_antihom}
\left(A\ts B\right)^\tau = B^\tau   A^\tau (-1)^{\bar{A}\bar{B}}\quad\text{for all homogenous }A,B\in \ndo\CC^{M|N}.
\eeq
Finally, the supertrace satisfies the identity
\beq\label{str_id}
\str A\ts B = \str A^\tau B^\tau  \quad\text{for all }A,B\in \ndo\CC^{M|N}.
\eeq
Indeed, consider \eqref{str_id} for  $A$ and $B$   homogenous. By \eqref{str_transposition}, the left-hand side 
is equal to 
$ \str (A\ts B)^\tau$, 
and by \eqref{str_cyclic}, the right-hand side is equal to
$ \str B^\tau   A^\tau (-1)^{\bar{A}\bar{B}}  $. Hence, the identity \eqref{tau_antihom} implies that both sides coincide.

The supertrace naturally extends to the tensor product spaces, e.g., let us consider
\beq\label{tensors}
\left(\ndo\CC^{M|N}\right)^{\ot n}\ot \dygtld .
\eeq
 For any $k=1,\ldots , n$, denote by $\str_k$ the action of the supertrace on the $k$-th tensor factor. Clearly, for all $x\in \dygtld$, we have
\begin{align*}
&\str_k e_{r_1\ts s_1}\ot \ldots \ot e_{r_n\ts s_n}\ot x\\ 
= \,\,& 
\delta_{i_k\ts j_k} (-1)^{\bar{i_k}}
  e_{r_1\ts s_1}\ot \ldots 
	\ot e_{r_{k-1}\ts s_{k-1}}\ot e_{r_{k+1}\ts s_{k+1}}
\ot \ldots\ot e_{r_n\ts s_n}\ot x.
\end{align*}
Let $\str_{1,\ldots ,n} =\str_1 \ldots \str_n$, so that the supertrace is applied on the tensor factors $1,\ldots ,n $ of \eqref{tensors}. The  formula  above implies 
$$
 \str_{1,\ldots ,n} e_{r_1\ts s_1}\ot \ldots \ot e_{r_n\ts s_n}\ot x 
=  
\delta_{i_1\ts j_1}\ldots \delta_{i_n\ts j_n} (-1)^{\bar{i_1}+\ldots +\bar{i_n}}
   x \quad\text{
 for all }x\in \dygtld .$$
Furthermore, the identity \eqref{str_id} can be generalized as
\beq\label{another_prop}
\str_{1,\ldots ,n} A\ts B = \str_{1,\ldots ,n} A^{\tau_1,\ldots ,\tau_n} B^{\tau_1,\ldots ,\tau_n}  \quad\text{for all }A,B\in \left(\ndo\CC^{M|N}\right)^{\ot n}.
\eeq

Next, we follow the exposition in \cite[Chap. 1]{Mnew} to recall the fusion procedure, which goes back to Jucys \cite{Juc}. Let $\mathfrak{S}_n$ be the symmetric group on the set $\left\{1,\ldots ,n\right\}$ and $\lambda  $ a partition of $n$. Hence, $\lambda=(\lambda_1,\ldots ,\lambda_m)\vdash n$ is a decreasing sequence of integers $\lambda_1 \geqslant\ldots \geqslant\lambda_m \geqslant 1$ such that $\lambda_1 +\ldots+\lambda_m =n$. We   identify $\lambda$ with its Young diagram.
Let 
$$
h(\lambda) =\prod_{(i,j)\in\lambda} (\lambda_i+\lambda_j^\prime -i-j+1),
$$ 
where $\lambda_j^\prime$ is the number of boxes in the column $j$ of $\lambda$.
The group algebra $\CC[\mathfrak{S}_n]$ is isomorphic to the direct sum of matrix algebras
$$
\CC[\mathfrak{S}_n]\cong\bigoplus_{\lambda\vdash n} M_{f_\lambda}(\CC),\quad\text{where }f_\lambda =  \frac{n! }{ h(\lambda)} 
$$ 
is the number of standard tableaux of shape $\lambda$,
 $M_k(\CC)$ denotes the algebra of all complex $k\times k$ matrices and the sum goes over all partitions $\lambda$ of $n$. For  any standard tableaux $\mathcal{U}$ and $\mathcal{U}^\prime$ denote by $e_{\mathcal{U}\ts\mathcal{U}^\prime}$ the matrix units in $M_{f_\lambda}(\CC)$. 
Note that the elements $e_{\mathcal{U} } = e_{\mathcal{U}\ts\mathcal{U} }$ are primitive idempotents of $\CC[\mathfrak{S}_n]$.
For any $1\leqslant i<j\leqslant n$, denote by $ (i,j)$ the transposition of the elements $i$ and $j$.
 Choose any standard tableau $\mathcal{U}$ of shape $\lambda\vdash n$ with entries $1,\ldots ,n$. 
Let 
\beq\label{contents}
 c_a =c_a(\mathcal{U}) =j-i\quad\text{if the element $a=1,\ldots ,n$ occupies the  box $(i,j)$ in $\mathcal{U}$.}
\eeq
Consider the rational function in complex variables $u_1,\ldots ,u_n$ with values in the group algebra $\CC[\mathfrak{S}_n][[h]]$  over $\CC[[h]]$ defined by
\beq\label{phi_super}
\phi(u_1,\ldots ,u_n)=
\prod_{1\leqslant i<j\leqslant n}\left(1-\frac{h(i,j)}{u_i -u_j}\right),
\eeq
 where the product is taken with respect to  the     lexicographical order on the set of pairs $(i,j)$ with $1\leqslant i<j\leqslant n$.
By the fusion procedure \cite[Prop. 1.1.7]{Mnew}, we have
\beq\label{super_fussion}
e_{\mathcal{U}} = \frac{1}{h(\lambda)}\phi(u_1,\ldots ,u_n)
\big|_{u_1=hc_1}\big|_{u_2=hc_2}\ldots \big|_{u_n=hc_n}.
\eeq

By \cite[Subsect. A.10]{Gow}, there exists a unique action of the 
  symmetric group 		$\mathfrak{S}_n$ on the tensor product
$(\CC^{M|N})^{\ot n}$ such that 
\beq\label{super_acction}
 (i,i+1)\mapsto P_{i\ts i+1}\quad\text{for all }i=1,\ldots ,n-1 ,
\eeq
where 
		$P_{i\ts i+1}$ is the action of the permutation operator $P$, as  given by \eqref{R}, on the tensor factors $i$ and $i+1$.
		Denote by $E_\mathcal{U}$ the image of the primitive idempotent $e_\mathcal{U}$, associated with the standard tableau $\mathcal{U} $ of shape $\lambda\vdash n$, under this action.

		In the next two lemmas, we recall two properties of the idempotent $E_U$, which will be essential in the proof of Theorem \ref{super_center} below. Recall that the integers $c_a=c_a(\mathcal{U})$ with $a=1,\ldots ,n$ are defined by \eqref{contents}.

\begin{lem}
The element $E_\mathcal{U}$ satisfies
\beq\label{eu_ybe}
E_\mathcal{U}\ts \R_{01}^-\ldots \R_{0n}^-=\R_{0n}^-\ldots \R_{01}^-\ts E_\mathcal{U}\Fand
E_\mathcal{U}\ts \R_{0n}^+\ldots \R_{01}^+=\R_{01}^+\ldots \R_{0n}^+\ts E_\mathcal{U},
\eeq
where
$$ 
\R_{0i}^-=\R_{0i} (x-hc_i )^{-1}\fand
\R_{0i}^+=\R_{0i} (x-hc_i ) \qquad\text{for }i=1,\ldots ,n 
$$
and the $R$-matrix $\R(u)$ is given by \eqref{hadicrmatrix}.
\end{lem}

\begin{prf}
Define
\beq\label{ruofn_super}
R(u_1,\ldots ,u_n) =\prod_{1\leqslant i<j\leqslant n} R_{ij}(u_i-u_j),
\eeq
where the product is taken with respect to  the     lexicographical order on the set of pairs $(i,j)$ with $1\leqslant i<j\leqslant n$ and the $R$-matrix $ R(u)$ is given by \eqref{hadicrmatrix}. 
Note that $R(u_1,\ldots ,u_n)$ equals the image of \eqref{phi_super} under the action \eqref{super_acction}.
By the Yang--Baxter equation \eqref{ybe}, we have
$$
R(u_1,\ldots ,u_n)\ts \R_{0n}(x-u_n)\ldots \R_{01}(x-u_1)=
\R_{01}(x-u_1)\ldots \R_{0n}(x-u_n)\ts R(u_1,\ldots ,u_n),
$$
where $R(u_1,\ldots ,u_n)$ is applied on the tensor factors $1,\ldots ,n$. 
Due to the fusion procedure \eqref{super_fussion}, by applying the consecutive evaluations 
$u_1=hc_1,\ldots ,u_n=hc_n$ to the identity above, we get
$$
E_{\mathcal{U}}\ts \R_{0n}(x-hc_n)\ldots \R_{01}(x-hc_1)=
\R_{01}(x-hc_1)\ldots \R_{0n}(x-hc_n)\ts E_{\mathcal{U}},
$$
which is exactly the second equality in \eqref{eu_ybe}. The first equality follows analogously.
\end{prf}

\begin{lem}
The element $E_\mathcal{U}$ satisfies
\beq\label{eu_rtt}
E_\mathcal{U}\ts T_{1}^+\ldots T_{n}^+ =T_{n}^+\ldots T_{1}^+\ts E_\mathcal{U}\Fand
E_\mathcal{U}\ts T_{n}^-\ldots T_{1}^-=T_{1}^-\ldots T_{n}^- \ts E_\mathcal{U},
\eeq
where
$$ 
T_{ i}^+=T_{ i\ts n+1}^+ (y+hc_i ) \fand
T_{ i}^-=T_{ i\ts n+1}^+ (y+hc_i )^{-1} \qquad\text{for }i=1,\ldots ,n.
$$
\end{lem}

\begin{prf}
The defining relation \eqref{hadic_rtt12} for the double Yangian implies
$$
R(u_1,\ldots ,u_n)\ts T_{1\ts n+1}^+(u_1)\ldots T_{n\ts n+1}^+(u_n)=
T_{n\ts n+1}^+(u_n)\ldots T_{1\ts n+1}^+(u_1)\ts R(u_1,\ldots ,u_n),
$$
where $R(u_1,\ldots ,u_n)$, as defined by \eqref{ruofn_super}, is applied on the tensor factors $1,\ldots ,n$. 
Due to the fusion procedure \eqref{super_fussion}, by applying the consecutive evaluations 
$u_1=y+hc_1,\ldots ,u_n=y+hc_n$ to the identity above, we get
$$
E_\mathcal{U}\ts T_{1\ts n+1}^+(y+hc_1)\ldots T_{n\ts n+1}^+(y+hc_n)=
T_{n\ts n+1}^+(y+hc_n)\ldots T_{1\ts n+1}^+(y+hc_1)\ts E_\mathcal{U} ,
$$
which is exactly the first equality in \eqref{eu_rtt}. The first equality follows analogously.
\end{prf}

		Set $\kappa = (N-M)/2$. 
		Introduce the formal power series
		$T_\mathcal{U} (u)\in \dygtld[[u^{\pm 1}]]$ by
\beq\label{telambda}
T_\mathcal{U} (u)=\str_{1,\ldots ,n}
E_\mathcal{U}\ts
T_1^+(u+hc_1)\ldots T_n^+(u+hc_n)\ts 
T_n^-(u+h(c_n+\kappa))^{-1}\ldots T_1^-(u+h(c_1+\kappa))^{-1}.
\eeq

\begin{thm}\label{super_center}
All coefficients of $T_\mathcal{U}(u)$ belong to the center of  $\dygtld$.
\end{thm}

\begin{prf}
It suffices to show that
$$
T^+(z)\ts T_\mathcal{U}(u)=T_\mathcal{U}(u)\ts T^+(z)\fand
T^-(z)\ts T_\mathcal{U}(u)=T_\mathcal{U}(u)\ts T^-(z).
$$
Both identities can be verified by arguing as in the proof of \cite[Thm. 4.4]{JKMY}, which is an even counterpart of this theorem. However, for completeness, we   present the proof of the second identity  in full detail.

Let
$$
\mathcal{T}_{[n]}(u) = T_1^+(u+hc_1)\ldots T_n^+(u+hc_n)\ts 
T_n^-(u+h(c_n+\kappa))^{-1}\ldots T_1^-(u+h(c_1+\kappa))^{-1},
$$
so that the series \eqref{telambda} can be written more briefly as
$$
T_\mathcal{U}(u)=\str_{1,\ldots ,n}
E_\mathcal{U}\ts
\mathcal{T}_{[n]}(u).
$$
The defining relations \eqref{hadic_rtt12} and \eqref{hadic_rtt3} for the double Yangian imply 
\begin{align*}
&T_0^-(z)\ts T_1^+(u) =
\R_{01}(z-u+h\kappa)^{-1} T_1^+(u)\ts T_0^-(z)\ts \R_{01}(z-u-h\kappa),\\
&T_0^-(z)\ts \R_{01}(z-u-h\kappa)\ts T_1^-(u+h\kappa)^{-1} =
T_1^-(u+h\kappa)^{-1}\ts \R_{01}(z-u-h\kappa)\ts T_0^-(z).
\end{align*}
By using these identities, we find
$$
 T_0^-(z)\ts \mathcal{T}_{[n]}(u)
= 
\R_{01}^-\ldots \R_{0n}^-\ts
\mathcal{T}_{[n]}(u)\ts
\R_{0n}^+\ldots \R_{01}^+\ts
T_0^-(z),
$$
where
$$ 
\R_{0i}^-=\R_{0i} (z-u+h(\kappa-c_i))^{-1}\Fand
\R_{0i}^+=\R_{0i} (z-u-h(\kappa+c_i)) 
$$
 for $i=1,\ldots ,n$. Hence, we need to prove that
\beq\label{super_tmp_01}
\str_{1,\ldots ,n}
E_\mathcal{U}\ts \R_{01}^-\ldots \R_{0n}^-\ts
\mathcal{T}_{[n]}(u)\ts
\R_{0n}^+\ldots \R_{01}^+
=
\str_{1,\ldots ,n}
E_\mathcal{U}\ts  
\mathcal{T}_{[n]}(u) .
\eeq

By using the first identity in \eqref{eu_ybe} for $x=z-u+h\kappa$, we find
$$
E_\mathcal{U}\ts \R_{01}^-\ldots \R_{0n}^-
=  \R_{0n}^-\ldots \R_{01}^-\ts E_\mathcal{U} 
=\R_{0n}^-\ldots \R_{01}^-\ts E_\mathcal{U}^2
=E_\mathcal{U} \ts \R_{0n}^-\ldots \R_{01}^-\ts E_\mathcal{U} 
$$
Hence, the left-hand side in \eqref{super_tmp_01} is equal to 
$$
\str_{1,\ldots ,n}
E_\mathcal{U}\ts \R_{01}^-\ldots \R_{0n}^-\ts E_\mathcal{U}\ts
\mathcal{T}_{[n]}(u)\ts
\R_{0n}^+\ldots \R_{01}^+.
$$
As the element $E_\mathcal{U}$ is even, we can use the property \eqref{str_cyclic} of the supertrace to write the expression above as
$$
\str_{1,\ldots ,n}
 \R_{01}^-\ldots \R_{0n}^-\ts E_\mathcal{U}\ts
\mathcal{T}_{[n]}(u)\ts
\R_{0n}^+\ldots \R_{01}^+\ts E_\mathcal{U}.
$$
Next, we use the equalities \eqref{eu_rtt} for $y=u$ and $y=u+h\kappa$ and the second equality in \eqref{eu_ybe} for $x=z-u-h\kappa$ to move the left copy of 
$ E_\mathcal{U}$ to the right, thus getting
$$
\str_{1,\ldots ,n}
 \R_{01}^-\ldots \R_{0n}^-\ts  
\mathcal{T}_{[n]}(u)^\prime\ts
\R_{01}^+\ldots \R_{0n}^+\ts E_\mathcal{U}^2
=
\str_{1,\ldots ,n}
 \R_{01}^-\ldots \R_{0n}^-\ts  
\mathcal{T}_{[n]}(u)^\prime\ts
\R_{01}^+\ldots \R_{0n}^+\ts E_\mathcal{U}
,
$$
where
$$
\mathcal{T}_{[n]}(u)^\prime = 
T_n^+(u+hc_n)\ldots T_1^+(u+hc_1)\ts 
T_1^-(u+h(c_1+\kappa))^{-1}\ldots T_n^-(u+h(c_n+\kappa))^{-1}.
$$
Finally, we use the same properties of  $E_\mathcal{U}$ to move its remaining copy  to the left, so that we get
$$ 
\str_{1,\ldots ,n}
 \R_{01}^-\ldots \R_{0n}^-\ts E_\mathcal{U}\ts  
\mathcal{T}_{[n]}(u) \ts
\R_{0n}^+\ldots \R_{01}^+ 
.
$$

By   the property \eqref{another_prop} of the supertrace, we conclude that this is   equal to
\begin{align*}
&\str_{1,\ldots ,n}
 \left(\R_{01}^-\ldots \R_{0n}^-\right)^{\tau_1\ldots \tau_n} \left(E_\mathcal{U}\ts  
\mathcal{T}_{[n]}(u) \ts
\R_{0n}^+\ldots \R_{01}^+ \right)^{\tau_1\ldots \tau_n}\\
=& 
\str_{1,\ldots ,n}
  (\R_{01}^-)^{\tau_1 } \ldots (\R_{0n}^-)^{ \tau_n} 
	(\R_{0n}^+)^{  \tau_n}\ldots (\R_{01}^+  )^{\tau_1 }
	\left(E_\mathcal{U}\ts  
\mathcal{T}_{[n]}(u) \right)^{\tau_1\ldots \tau_n} .
\end{align*}
The $R$-matrices $(\R_{0i}^-)^{ \tau_i} 
	$ and $(\R_{0i}^+)^{  \tau_i}$ cancel,  due to the crossing symmetry  \eqref{hadic_csym}.
	Therefore, by \eqref{str_transposition},
  the expression above is equal to
	$$
\str_{1,\ldots ,n}
	\left(E_\mathcal{U}\ts  
\mathcal{T}_{[n]}(u) \right)^{\tau_1\ldots \tau_n} 
=
\str_{1,\ldots ,n}
	 E_\mathcal{U}\ts  
\mathcal{T}_{[n]}(u)  
=T_\mathcal{U}(u),
$$
as required.
\end{prf}
	
At the end of this section, we present two simple applications of Theorem \ref{super_center}. 
 Define the {\em center} of   $h$-adic quantum vertex superalgebra $V$ by
$$
\mathfrak{z}(V)=\left\{v\in V\,:\,Y(u,z)v\in V[[z]]\text{ for all }u\in V\right\};
$$
see \cite{DGK,JKMY} for more information on the notion of center in quantum vertex algebra theory. 
Note that, by the form of the vertex operator map \eqref{hadic_y},   all elements $v\in V^c(\glmn)$ which satisfy $T^-(u)v=I\ot v$ belong to the center 
$\mathfrak{z}(V^c(\glmn))$.
Denote by $V^{crit}(\glmn)$ the $h$-adic quantum vertex superalgebra $V^c(\glmn)$ at the critical level $c=N-M$.

\begin{kor}
All coefficients of the series
\beq\label{qvoa_telambda}
T^+_\mathcal{U}(u)\coloneqq\str_{1,\ldots ,n}
E_\mathcal{U}\ts
T_1^+(u+hc_1)\ldots T_n^+(u+hc_n)\vac \in V^{crit}(\glmn)  [[u]].
\eeq 
belong to the center of $V^{crit}(\glmn)$.
\end{kor}

\begin{prf}
The corollary follows by the well-known  argument, 
$$
T^-(u)\ts T^+_\mathcal{U}(u)
=T^-(u)\ts T_\mathcal{U}(u)1
=T_\mathcal{U}(u) \ts T^-(u)  1
=T_\mathcal{U}(u)    1
=T^+_\mathcal{U}(u), 
$$
which relies on Theorem \ref{super_center}.
\end{prf}

The next corollary  uses the identification of the $\CC[[h]]$-module 
$V^{crit}(\glmn)$ with the $h$-adically completed dual Yangian   $\Yd_h$, so that the coefficients of the series $T^+_\mathcal{U}(u)$, given by \eqref{qvoa_telambda}, are  regarded
as elements of   $\Yd_h$.

\begin{kor}\label{super_kom}
The coefficients of all $T^+_\mathcal{U}(u)$ generate a commutative subalgebra of the $h$-adically completed dual Yangian   $\Yd_h$.
\end{kor}

\begin{prf}
The corollary is proved by the usual argument which we recall for completeness; see, e.g., the proofs of \cite[Cor. 4.4]{JLZ} or \cite[Cor. 4.6]{JKMY}. For any two standard tableaux $\mathcal{U}$ and $\mathcal{V}$ with entries $1,\ldots, n$ and $1,\ldots ,m$ respectively, by using Theorem \ref{super_center}, we find
$$
T_\mathcal{U}^+(u)\ts T_\mathcal{V}^+(v) =
T_\mathcal{U}^+(u)\ts T_\mathcal{V} (v)\vac =
T_\mathcal{V} (v)\ts T_\mathcal{U}^+(u)  \vac =
T_\mathcal{V} (v)\ts T_\mathcal{U} (u)  \vac  
$$
and, analogously, 
$T_\mathcal{V}^+(v)\ts T_\mathcal{U}^+(u) = 
T_\mathcal{U} (u)\ts T_\mathcal{V} (v)  \vac$. This implies
$T^+_\mathcal{V} (v)\ts T^+_\mathcal{U} (u)=T^+_\mathcal{U} (u)\ts T^+_\mathcal{V} (v)  $, as required.
\end{prf}

 \section{Central elements in the completed reflection algebra at the critical level}\label{section_last}

Let $G=(g_{ij}) $ be the diagonal matrix 
\eqref{gmatricca}. 
From now on, we consider the   algebra $\DBB$ over the commutative ring $\CC[[h]]$. As with the double Yangian, its definition, given in Subsection \ref{subsection_32}, is translated to the $h$-adic setting by formal rescaling $u\mapsto u/h$ of the variables and generators
$$
\beta_{ij}^{(r)}\mapsto h^{r-1}\ts \beta_{ij}^{(r)},\qquad
\beta_{ij}^{(-r)}\mapsto h^{-r }\ts \beta_{ij}^{(-r)}\Fand
\mathcal{C}\mapsto \mathcal{C}
$$
for all $r=1,2,\ldots  .$ We consider the $R$-matrices $R(u)$ and $\R(u)$ given by
 \eqref{hadicrmatrix}.

The
{\em reflection algebra $\DBB$} is defined as the $\ZZ_2$-graded unital associative algebra over $\CC[[h]]$ generated by the elements $\mathcal{C}$ and $\beta_{ij}^{(\pm r)}$  with $  i,j=1,\ldots , M+N$ and $r=1,2,\ldots,$ 
where $\mathcal{C}$ is   even and central   and  the parity of
  $\beta_{ij}^{(\pm r)}$  is  $\bar{i}+\bar{j}$.
The generators are
subject to the  defining relations which are written in terms of the  matrices 
$$
\mathcal{B}^{\pm }(u)=\sum_{i,j=1}^{M+N} (-1)^{\bar{i}\bar{j}+\bar{j}} e_{ij}\ot \beta^{\pm }_{ij}(u),
$$
 where
$$
\beta^-_{ij}(u)=g_{ij}+h\sum_{r\geqslant 1} \beta_{ij}^{(r)} u^{-r}
\Fand
\beta^+_{ij}(u)=g_{ij}-h\sum_{r\geqslant 1} \beta_{ij}^{(-r)} u^{r-1}.
$$
The defining relations  are given by  
\begin{align*}
R (u-v)\ts \mathcal{B}_1^\pm (u)\ts  R (u+v )\ts \mathcal{B}_2^\pm (v)
&= \mathcal{B}^\pm_2(v)\ts  R (u+v )\ts \mathcal{B}^\pm_1(u)\ts R (u-v),  \\
 \R (u-v +h\mathcal{C})\ts \mathcal{B}^-_1 (u)\ts  \R (u+v -h\mathcal{C})\ts \mathcal{B}_2^+ (v)
&= \mathcal{B}_2^+ (v)\ts  \R (u+v+h\mathcal{C} )\ts \mathcal{B}^-_1 (u)\ts  \R (u-v -h\mathcal{C}),  \\ 
\mathcal{B}^\pm (u)\ts \mathcal{B}^\pm (-u)&=1.
\end{align*}
  We assume that $\DBB$ is completed with respect to the $h$-adic topology.

Clearly, all elements   $\beta_{ij}^{(r)}$ (resp. $\beta_{ij}^{(-r)}$) with $r=1,2,\ldots$ belong to the Yangian $\Y$ (resp. $h$-adically completed   dual Yangian $\Yd_h$).  
Let $\BB$ (resp. $\BBd$) be the subalgebra of 
$\Y$ (resp.   $\Yd_h$) generated by all   $b_{ij}^{(r)}$ (resp. $b_{ij}^{(-r)}$) with $r=1,2,\ldots .$

Let us extend the notion of restricted module  to reflection algebras.
Suppose $W$ is
 a module over the superalgebra  $\DBB$. Then $W$ is said to be a {\em restricted module}, if 
it is topologically free as a $\CC[[h]]$-module and  
the action $\mathcal{B}^-(u)_W$ of the matrix $\mathcal{B}^-(u) $ belongs to $\ndo\CC^{M|N}\ot \om(W,W[u^{-1}]_h)$. Moreover, a $\DBB$-module $W$ is said to be of {\em level $c\in \CC$} if   the   central element $\mathcal{C}$ acts on $W$ as scalar multiplication by $c$.

\begin{ex}\label{ex_61}
The quotient of the algebra $\DBB$ over the $h$-adically completed left ideal generated by $\mathcal{C}-c$ and all elements $\beta_{ij}^{(r)}$ with $r\geqslant 1$ is naturally equipped with the structure of restricted $\DBB$-module of level $c$.  We shall denote it by $\BBdc$, as it is isomorphic, as a $\CC[[h]]$-module, to the $h$-adic completion of $\BBd$.
\end{ex}

For any positive integer $n$ and the variables $u=(u_1,\ldots ,u_n)$ define
$$
\mathcal{B}_{[n]}^\pm (u|z)= \prod_{i=1,\ldots ,n}^{\longrightarrow}  \left(\mathcal{B}_{i\ts n+1}^\pm(z+u_i)\ts R_{i\ts i+1}\ldots R_{i\ts n}\right),\quad\text{where }R_{rs}=R_{rs}(2z+u_r+u_s) 
$$
and the arrow indicates the order of factors. For example, we have
$$
\mathcal{B}_{[1]}^\pm (u|z)=\mathcal{B}_{1  2}^\pm (z+u_1)\fand
\mathcal{B}_{[2]}^\pm (u|z)=\mathcal{B}_{1  3}^\pm (z+u_1)\ts R_{12}(2z+u_1+u_2)\ts \mathcal{B}_{2  3}^\pm (z+u_2).
$$
Suppose $W$ is a restricted $\DBB$-module. Denote by $\mathcal{B}_{[n]}^\pm (u|z)_W$ the action of $\mathcal{B}_{[n]}^\pm (u|z)$ on $W$. We have
\begin{align*}
&\mathcal{B}_{[n]}^+ (u|z)_W\in(\ndo\CC^{M|N})^{\ot n}\ot \om (W,W((z))_h[[u_1,\ldots ,u_n]]),\\
&\mathcal{B}_{[n]}^- (u|z)_W\in(\ndo\CC^{M|N})^{\ot n}\ot \om (W,W[z^{-1}]_h[[u_1,\ldots ,u_n]]).
\end{align*}
Furthermore, if the variable $z$ is omitted, we write 
\beq\label{morenotation}
\mathcal{B}_{[n]}^\pm (u )= \prod_{i=1,\ldots ,n}^{\longrightarrow}  \left(\mathcal{B}_{i\ts n+1}^\pm( u_i)\ts R_{i\ts i+1}(u_i + u_{i+1})\ldots R_{i\ts n}( u_i+u_n)\right)  .
\eeq
Note that, in particular, we have $ \mathcal{B}_{[1]}^{\pm }(u) =\mathcal{B}^{\pm }(u) $

By using the above notation, along with  \eqref{RnM1} and  \eqref{RnM2}, the reflection equations  can be generalized to the identities of formal power series with coefficients in 
$$
\overbrace{(\ndo\mathbb{C}^{M|N})^{\ot  n}}^{1} \ot 
\overbrace{(\ndo\mathbb{C}^{M|N})^{\ot  m}}^{2}\ot
\overbrace{\DBB}^{3}  
$$
as follows.

\begin{lem}\label{refl_gen_lemma}
For any positive integers $n$ and $m$ and the families of variables $u=(u_1,\ldots ,u_n)$ and $v=(v_1,\ldots ,v_m)$ we have
\begin{align*}
&R_{nm}^{12} (u|v|z_1-z_2)\ts \mathcal{B}_{[n]}^{\pm\, 13} (u|z_1 )\ts  \wndr{R}_{nm}^{12} (u|v|z_1+z_2 )\ts \mathcal{B}_{[m]}^{\pm\,23} (v|z_2)\\
&\qquad\qquad= \mathcal{B}_{[m]}^{\pm\,23} (v|z_2)\ts  \wndr{R}_{nm}^{12} (u|v|z_1+z_2 )\ts \mathcal{B}_{[n]}^{\pm\, 13} (u|z_1 )\ts R_{nm}^{12} (u|v|z_1-z_2),  \\
 &\R_{nm}^{12} (u|v| z_1-z_2 +h\mathcal{C})\ts \mathcal{B}_{[n]}^{-\, 13} (u|z_1 )\ts  \wndr{\R}_{nm}^{12} (u|v|z_1+z_2 -h\mathcal{C})\ts \mathcal{B}_{[m]}^{+\,23}(v|z_2)\\
&\qquad\qquad =\mathcal{B}_{[m]}^{+\,23} (v|z_2)\ts  \wndr{\R}_{nm}^{12} (u|v|z_1+z_2 +h\mathcal{C})\ts \mathcal{B}_{[n]}^{-\, 13} (u|z_1 )\ts  \R_{nm}^{12} (u|v| z_1-z_2 -h\mathcal{C}).   
\end{align*}
\end{lem} 

In the next theorem, we use the notion of quasi module, given by Definition \ref{qvoamodule}, to establish a connection between restricted $\DBB$-modules and quasi modules over the $h$-adic quantum vertex superalgebra $V^c(\glmn)$, given by Theorem \ref{glavni_teorem}.

\begin{thm}\label{teorem_kvazi}
 Let W be a restricted $\DBB$-module of level $c\in\CC$.
There
exists a unique structure of quasi $V^{2c}(\glmn)$-module   
 on $W$ such that  
\beq\label{quasimodulemap}
Y_W(T_{[n]}^+(u ) \vac ,z)=
\mathcal{B}_{[n]}^+(u|z )_W\ts 
\mathcal{B}_{[n]}^-(u|z+hc )^{-1}_W.
\eeq
\end{thm}

\begin{prf}
The theorem can be verified  by lengthy but standard computations. 
They rely on the usual $R$-matrix techniques and Lemma \ref{refl_gen_lemma}   and they closely follow the proof of \cite[Thm. 2.7]{Koz}, which is the  even counterpart of this theorem. As with the second assertion of Theorem \ref{glavni_teorem}, the grading restriction  \eqref{super_03}, which is not present in \cite[Thm. 2.7]{Koz}, is an immediate consequence of the grading restriction in the definition of the notion of module over an associative superalgebra.
\end{prf}

For any $c\in\CC$, denote by $\DBBc$ the {\em reflection algebra at the level $c$}, i.e.,   the quotient of the algebra $\DBB$ over the $h$-adically completed ideal generated by $\mathcal{C}-c $. Let $J_p$ with $p\geqslant 1$ be the  $h$-adically completed  ideal of $\DBBc$ generated by all $\beta_{ij}^{(r)}$ with $r\geqslant p$ such that $h^m a\in J_p$ implies $  a\in J_p$ for all $m\geqslant 1$.
Define the completion of $\DBBc$ as the  inverse limit
$$
\DBBctld = \lim_{\longleftarrow} \DBBc  / J_p.
$$

From now on, we consider the completed reflection algebra at the {\em critical level} $c=(N-M)/2$, which we denote by $\DBBcrittld = \DBBNMtld$. 
Note that we refer to the value $ (N-M)/2$ as  critical   due to Theorem \ref{teorem_kvazi}, which, in particular,  implies that any  restricted $\DBB$-module $W$ of level $(N-M)/2$ is equipped with the structure of quasi $V^{c}(\glmn)$-module at the {\em critical level} $c=N-M$ via \eqref{quasimodulemap}. 
 By applying the corresponding quasi module map $Y_W(\cdot, z)$, as given by \eqref{quasimodulemap}, on the constant term $T^+_\mathcal{U} =T^+_\mathcal{U}(0)$ of the power series \eqref{qvoa_telambda}, we obtain
\beq\label{superywc}
Y_W( T^+_\mathcal{U} ,z)=
\str_{1,\ldots ,n}
E_\mathcal{U}\ts
\mathcal{B}_{[n]}^+(z_\mathcal{U} )_W\ts 
\mathcal{B}_{[n]}^-( z_\mathcal{U}+h(N-M)/2 )^{-1}_W 
  \in \ndo W  [[z^{\pm 1}]],
\eeq
where 
$$
z_\mathcal{U} + ha =(z+h(c_1+a),\ldots ,z+h(c_n +a)).
$$
Note that the coefficients of the series  \eqref{superywc},
\beq\label{belambdaseries}
\mathcal{B}_\mathcal{U}(z)=
 \str_{1,\ldots ,n}
E_\mathcal{U}\ts
\mathcal{B}_{[n]}^+(z_\mathcal{U} ) \ts 
\mathcal{B}_{[n]}^-( z_\mathcal{U}+h(N-M)/2)^{-1}
\eeq
can be also regarded as elements of the completed algebra $\DBBcrittld  $. 

\begin{thm}\label{central_thm_refl}
All coefficients of $\mathcal{B}_\mathcal{U}(z)$ belong to the center of $\DBBcrittld  $. 
\end{thm}

\begin{prf}
It suffices to show that
\beq\label{weneedtoprove}
\mathcal{B}^+(u)\ts \mathcal{B}_\mathcal{U}(z)=\mathcal{B}_\mathcal{U}(z)\ts \mathcal{B}^+(u)\fand
\mathcal{B}^-(u)\ts \mathcal{B}_\mathcal{U}(z)=\mathcal{B}_\mathcal{U}(z)\ts \mathcal{B}^-(u).
\eeq
Both identities can be verified by arguing as in the proof of \cite[Thm. 3.2]{Koz}, which is an even counterpart of this theorem. However, for completeness, we   provide some details of the proof.
Let
$$
\mathcal{B}_{[n]}(z) = \mathcal{B}_{[n]}^+(z_\mathcal{U} ) \ts 
\mathcal{B}_{[n]}^-( z_\mathcal{U}+h\kappa)^{-1},
$$
where $\kappa=(N-M)/2$,
so that the series \eqref{belambdaseries} can be written more briefly as
$$
\mathcal{B}_\mathcal{U}(z)=\str_{1,\ldots ,n}
E_\mathcal{U}\ts
\mathcal{B}_{[n]}(z).
$$

Consider the tensor product
$$
\overbrace{(\ndo\mathbb{C}^{M|N})^{\ot  n}}^{1} \ot 
\overbrace{ \ndo\mathbb{C}^{M|N} }^{2}\ot
\overbrace{\DBBcrittld }^{3}.  
$$
By using Lemma \ref{refl_gen_lemma} and the notation conventions \eqref{RnM3}, \eqref{RnM4} and \eqref{morenotation}, we find
\begin{align*}
&\mathcal{B}_{[1]}^{+\,23} (u)\ts  \wndr{\R}_{n1}^{12} (z_{\mathcal{U}}|u  )\ts \mathcal{B}_{[n]}^{+\, 13} ( z_{\mathcal{U}} )\ts \R_{n1}^{12} (z_{\mathcal{U}}|u  ),  \\
&\qquad\qquad= \R_{n1}^{12} (z_{\mathcal{U}}|u  )\ts \mathcal{B}_{[n]}^{+\, 13} (z_{\mathcal{U}} )\ts  \wndr{\R}_{n1}^{12} (z_{\mathcal{U}}|u   )\ts \mathcal{B}_{[1]}^{+\,23} (u),\\
&\mathcal{B}_{[1]}^{+\,23} (u)\ts   \R_{n1}^{12} (z_{\mathcal{U}}|u  )^{-1}\ts \mathcal{B}_{[n]}^{-\, 13} ( z_{\mathcal{U}}+h\kappa )^{-1}\ts \wndr{\R}_{n1}^{12} (z_{\mathcal{U}}+2h\kappa|u  )^{-1}\\
&\qquad\qquad= \wndr{\R}_{n1}^{12} (z_{\mathcal{U}}|u  )^{-1}\ts \mathcal{B}_{[n]}^{-\, 13} (z_{\mathcal{U}}+h\kappa )^{-1}\ts   \R_{n1}^{12} (z_{\mathcal{U}}+2h\kappa|u   )^{-1}\ts \mathcal{B}_{[1]}^{+\,23} (u).   
\end{align*}
Next, by combining these two identities, we obtain
\begin{align*}
&\mathcal{B}_{[1]}^{+\,23} (u)\ts  \wndr{\R}_{n1}^{12} (z_{\mathcal{U}}|u  )\ts \mathcal{B}_{[n]}^{  13} ( z )\ts \wndr{\R}_{n1}^{12} (z_{\mathcal{U}}+2h\kappa|u  )^{-1}\\
&\qquad\qquad =\R_{n1}^{12} (z_{\mathcal{U}}|u  )\ts \mathcal{B}_{[n]}^{  13} ( z )\ts \R_{n1}^{12} (z_{\mathcal{U}}+2h\kappa|u   )^{-1} \ts \mathcal{B}_{[1]}^{+\,23} (u).
\end{align*}
Finally,    applying the idempotent $E_{\mathcal{U}}$ on the tensor factors $1,\ldots ,n$ of the above equality and then taking the supertrace over the same tensor factors yields
\begin{align*}
&
\mathcal{B}^{+} (u) \left(
\str_{1,\ldots ,n}
E_\mathcal{U}\ts \wndr{\R}_{n1}^{12} (z_{\mathcal{U}}|u  )\ts \mathcal{B}_{[n]}^{  13} ( z )\ts \wndr{\R}_{n1}^{12} (z_{\mathcal{U}}+2h\kappa|u  )^{-1}\right)\\
= & \left(\str_{1,\ldots ,n}
E_\mathcal{U}\ts\R_{n1}^{12} (z_{\mathcal{U}}|u  )\ts \mathcal{B}_{[n]}^{  13} ( z )\ts \R_{n1}^{12} (z_{\mathcal{U}}+2h\kappa|u   )^{-1} \right) \mathcal{B}^{+} (u).
\end{align*}
Therefore, to verify the first identity in \eqref{weneedtoprove}, it suffices to show that both expressions inside the brackets  coincide with the series \eqref{belambdaseries}, i.e., to prove  the equalities
\begin{align*}
 \mathcal{B}_\mathcal{U}(z)\ot I & =
\str_{1,\ldots ,n}
E_\mathcal{U}\ts \wndr{\R}_{n1}^{12} (z_{\mathcal{U}}|u  )\ts \mathcal{B}_{[n]}^{  13} ( z )\ts \wndr{\R}_{n1}^{12} (z_{\mathcal{U}}+2h\kappa|u  )^{-1}\\
 &  =
\str_{1,\ldots ,n}
E_\mathcal{U}\ts\R_{n1}^{12} (z_{\mathcal{U}}|u  )\ts \mathcal{B}_{[n]}^{  13} ( z )\ts \R_{n1}^{12} (z_{\mathcal{U}}+2h\kappa|u   )^{-1} .
\end{align*}
The equalities can be proved by arguments similar to those in the corresponding part of the proof of Theorem \ref{super_center}, which shows that \eqref{super_tmp_01} holds. As before, the arguments rely on  the crossing symmetry properties \eqref{hadic_csym} and the fusion procedure \eqref{super_fussion}.

As for the second identity in \eqref{weneedtoprove}, it is more convenient to consider  the tensor product
$$
\overbrace{ \ndo\mathbb{C}^{M|N} }^{1} \ot 
\overbrace{(\ndo\mathbb{C}^{M|N})^{\ot  n}}^{2}\ot
\overbrace{\DBBcrittld }^{3}.  
$$
As before, by using Lemma \ref{refl_gen_lemma}, one obtains the equalities
\begin{align*}
&
\R_{1n}^{12}(u+h\kappa|z_{\mathcal{U}})\ts
\mathcal{B}_{[1]}^{-\,13} (u)\ts 
 \wndr{\R}_{1n}^{12} (u-h\kappa|z_{\mathcal{U}}  )\ts 
\mathcal{B}_{[n]}^{+\, 23} ( z_{\mathcal{U}} ) ,  \\
&\qquad\qquad= 
\mathcal{B}_{[n]}^{+\, 23} ( z_{\mathcal{U}} )\ts 
 \wndr{\R}_{1n}^{12} (u+h\kappa|z_{\mathcal{U}}  )\ts 
\mathcal{B}_{[1]}^{-\,13} (u)\ts 
\R_{1n}^{12}(u-h\kappa|z_{\mathcal{U}}),\\
&\mathcal{B}_{[1]}^{-\,13} (u)\ts
\R_{1n}^{12}(u-h\kappa|z_{\mathcal{U}})\ts
\mathcal{B}_{[n]}^{-\, 23} ( z_{\mathcal{U}}+h\kappa )^{-1}\ts 
 \wndr{\R}_{1n}^{12} (u+h\kappa|z_{\mathcal{U}}  )^{-1}\\ 
&\qquad\qquad= \wndr{\R}_{1n}^{12} (u+h\kappa|z_{\mathcal{U}}  )^{-1}\ts 
\mathcal{B}_{[n]}^{-\, 23} ( z_{\mathcal{U}}+h\kappa )^{-1}\ts 
\R_{1n}^{12}(u-h\kappa|z_{\mathcal{U}})\ts
\mathcal{B}_{[1]}^{-\,13} (u),   
\end{align*}
which   imply the identity
\begin{align*}
&\mathcal{B}_{[1]}^{-\,13} (u)\ts  \wndr{\R}_{1n}^{12} (u-h\kappa |z_{\mathcal{U}}  )\ts \mathcal{B}_{[n]}^{  23} ( z )\ts \wndr{\R}_{1n}^{12}(u+h\kappa |z_{\mathcal{U}}  )^{-1}\\
&\qquad\qquad =
\R_{1n}^{12}(u+h\kappa |z_{\mathcal{U}}  )^{-1}\ts 
\mathcal{B}_{[n]}^{  23} ( z )\ts
\R_{1n}^{12}(u-h\kappa |z_{\mathcal{U}}  ) \ts 
\mathcal{B}_{[1]}^{-\,13} (u)
.
\end{align*}
Applying   the idempotent $E_{\mathcal{U}}$ on the tensor factors $2,\ldots ,n+1$ of the above equality and then taking the supertrace over the same tensor factors
 yields
\begin{align*}
&\mathcal{B}^{- } (u)
\left(
\str_{2,\ldots ,n+1}
E_\mathcal{U}\ts   \wndr{\R}_{1n}^{12} (u-h\kappa |z_{\mathcal{U}}  )\ts \mathcal{B}_{[n]}^{  23} ( z )\ts \wndr{\R}_{1n}^{12}(u+h\kappa |z_{\mathcal{U}}  )^{-1}\right)\\
 = & 
\left(
\str_{2,\ldots ,n+1}
E_\mathcal{U}\ts 
\R_{1n}^{12}(u+h\kappa |z_{\mathcal{U}}  )^{-1}\ts 
\mathcal{B}_{[n]}^{  23} ( z )\ts
\R_{1n}^{12}(u-h\kappa |z_{\mathcal{U}}  ) \right)
\mathcal{B}^{- } (u)
.
\end{align*}
Therefore, to verify the second identity in \eqref{weneedtoprove}, it suffices to show that both expressions inside the brackets  coincide with   $I\ot \mathcal{B}_\mathcal{U}(z) $, which can be again done by arguing as in the corresponding part of the proof of  Theorem \ref{super_center}.
\end{prf}

At the end, we single out a simple consequence of Theorem \ref{central_thm_refl}. Consider the $\DBB$-module $\BBdcrit=\BBdNM $ from Example \ref{ex_61}. By applying   $\mathcal{B}_\mathcal{U}(z)$, defined by \eqref{belambdaseries}, on $1\in \BBdcrit$, we obtain the series in $ \BBdcrit((z))_h $,
$$
\mathcal{B}^+_\mathcal{U}(z)=
\mathcal{B}_\mathcal{U}(z)1.
$$
Note that its coefficients can be also regarded as elements of the $h$-adically completed algebra $\BBd_h$, so that  
$\mathcal{B}^+_\mathcal{U}(z)$ belongs to $ \BBd_h((z))_h$. The following corollary can be verified by the argument which follows the proof of Corollary \ref{super_kom} and employs  Theorem \ref{central_thm_refl}.

\begin{kor} 
The coefficients of all $\mathcal{B}^+_\mathcal{U}(z)$ generate a commutative subalgebra of   $\BBd_h$.
\end{kor}

\section*{Acknowledgement} 
L. B. is member of Gruppo Nazionale per le Strutture Algebriche, Geometriche e le loro Applicazioni  (GNSAGA) of the Istituto Nazionale di Alta Matematica (INdAM) and part of the project MMNLP (Mathematical Methods in Non Linear Physics) of INFN.
L. B. was partially supported by the project Representation Theory and Applications, Bando Ateneo 2023 of Sapienza University of Rome.
S. K. is partially supported by the Croatian Science Foundation under the project IP-2025-02-4720  and by the project ``Implementation of cutting-edge research and its application as part of the Scientific Center of Excellence for Quantum and Complex Systems, and Representations of Lie Algebras'', Grant No. PK.1.1.10.0004, co-financed by the European Union through the European Regional Development Fund - Competitiveness and Cohesion Programme 2021--2027.
This research was supported by the European Union -- NextGenerationEU through the National Recovery and Resilience Plan 2021-2026. Institutional grant of University of Zagreb Faculty of Science Strengthening scientific production, international presence and social impact of mathematical research (IK IA 1.1.3. Impact4Math).

\end{document}